\documentclass[oneside,final]{article}
\usepackage[right = 3.25 cm,left = 3.25 cm]{geometry}

\usepackage{amsthm,amsmath,amssymb,amsfonts,stmaryrd}
\usepackage{graphicx, tikz, pgfplots}
\usepgfplotslibrary{groupplots}
\usepackage{hhline}
\usepackage{algorithm,algpseudocode}
\usepackage{multirow}
\usepackage{appendix}
\usepackage{placeins}
\usepackage{mathtools}
\mathtoolsset{showonlyrefs=true}
\usepackage{fixltx2e}
\MakeRobust{\eqref}
\usepackage{bigstrut}
\usepackage{subcaption}
\usepackage{comment}
\usepackage{siunitx}
\usepackage{enumitem}
\usepackage{cancel}

\newcommand{\jump}[1]{\left[\mkern-1.5mu \left[#1\right] \mkern-1.5mu\right]}
\newcommand{\avg}[1]{\left\{ \mkern-5mu \left\{#1 \right\} \mkern-5mu \right\}}

\include{MACRO}
\graphicspath{{./IMG/}}

\theoremstyle{plain}
\newtheorem{theorem}{Theorem}
\newtheorem{lemma}[theorem]{Lemma}
\newtheorem{remark}[theorem]{Remark} 
\newtheorem{assumption}[theorem]{Assumption}
\newtheorem{definition}[theorem]{Definition}
\newtheorem{proposition}[theorem]{Proposition}

\begin{document}
\title{Discontinuous Galerkin approximation of the fully-coupled thermo-poroelastic problem \footnotemark[1]}

\author{ Paola F. Antonietti\footnotemark[2] \ \footnotemark[3] \and 
Stefano Bonetti \footnotemark[2] \ \footnotemark[4] \and Michele Botti \footnotemark[2] \  \footnotemark[5]}

\maketitle
\renewcommand{\thefootnote}{\fnsymbol{footnote}}
\footnotetext[1]{\textbf{Funding:} This work has received funding from the European Union’s Horizon 2020 research and innovation programme under the Marie Skłodowska-Curie grant agreement No. 896616   (project PDGeoFF).
P.F.A. has been partially funded by the research grants PRIN2017 n. 201744KLJL and PRIN2020 n. 20204LN5N5 funded by the Italian Ministry of Universities and Research (MUR). S.B., M.B., and P.F.A. are members of INdAM-GNCS.}
\footnotetext[2]{MOX-Laboratory for Modeling and Scientific Computing, Dipartimento di Matematica, Politecnico di Milano, Piazza Leondardo da Vinci 32, 20133 Milano, Italy.}
\footnotetext[3]{Email: paola.antonietti@polimi.it}
\footnotetext[4]{Email: stefano.bonetti@polimi.it}
\footnotetext[5]{Email: michele.botti@polimi.it}
\renewcommand{\thefootnote}{\arabic{footnote}}

\begin{abstract} We present and analyze a discontinuous Galerkin method for the numerical modelling of the non-linear fully-coupled thermo-poroelastic problem. For the spatial discretization, we design a high-order discontinuous Galerkin method on polygonal and polyhedral grids based on a novel four-field formulation of the problem. To handle the non-linear convective transport term in the energy conservation equation we adopt a fixed-point linearization strategy. We perform a robust stability analysis for the linearized semi-discrete problem under mild requirements on the problem data. A priori $hp$-version error estimates in suitable energy norms are also derived. A complete set of numerical simulations is presented in order to validate the theoretical analysis, to inspect numerically the robustness properties, and to test the capability of the proposed method in a practical scenario inspired by a geothermal problem.
\end{abstract}

\section{Introduction}
\label{sec:introduction}
Poroelasticity inspects the interaction among fluid flow and elastic deformations within a porous medium and finds its origin in the works of Biot \cite{Biot1941} and Terzaghi \cite{Terzaghi1943}. In several applications in the context of human geological activities, such as geothermal energy production and $CO_2$ sequestration, the temperature plays a key role in the description of the physical phenomena. In order to correctly describe these subsurface processes, the model should also take into account the influence of the temperature on the fluid flow and mechanical deformation, leading to a fully-coupled thermo-poroelastic (TPE) system of equations. 

In the framework of geosciences applications, the subsoil is modelled as a fully-saturated poroelastic material under the additional assumptions of small deformations and quasi-static regime. The TPE model derived in \cite{Brun2018} through a two-scale expansion \cite{Hornung1997} consists of three equations expressing the conservation of mass, momentum, and energy. For what concerns the first two equations, we can recognize the structure of the Biot's  system with an additional contribution of the temperature. In the conservation of mass, the temperature affects the fluid content, i.e. the amount of fluid that can be injected into a fixed control volume. In the momentum conservation equation, we observe the interplay among the Cauchy stress for the elastic skeleton and the isotropic terms coming from the pressure and the temperature. A complete study on the constitutive laws for the total stress tensor, the rate of change of fluid mass, and the rate of change of energy are presented in \cite{Coussy2003}. An alternative formulation of the TPE model is discussed in \cite{VanDuijn2019} under different assumptions on the deformations rates.
In the general case, the energy conservation equation takes into account both the conductive and convective contributions. Properly handling the convective term is one of the main challenge, since it yields an additional non-linear coupling between fluid flow and heat flux. We highlight that, for small values of the Péclet number, this term can be neglected, e.g. \cite{Gatmiri1997,Lee1997}, where it is assumed that the energy is balanced only by conduction. In our analysis, we tackle this term by the use of a proper iterative linearization procedure.

In \cite{Brun2019}, the well-posedness of the fully-coupled TPE model is proved by writing the problem in a six-field mixed form. The aim of this paper is to present and study a novel formulation with the introduction of only one additional (scalar) equation, leading to a numerical method that is more appealing from the computational point of view. Moreover, the introduction of the pseudo-total pressure variable ensures inf-sup stability and robustness with respect to locking phenomena in the quasi-incompressible limit. In \cite{Botti2020, Lee2017, Oyarzua2016} an analogous approach is considered for the poroelastic problem. Note that, a further possibility for dealing with the quasi-incompressible case in the analysis of the Biot's system is the introduction of the solid pressure (cf. \cite{Haga2012, Riviere2017}). In \cite{Khan2022}, a theoretical investigation on the advantages of considering the pseudo-total pressure is presented.

The spatial discretization of the aforementioned problem is set into the framework of the discontinuous Galerkin (DG) finite element methods. Examples of the application of DG schemes can be found for second-order elliptic problems \cite{Antonietti2013,Bassi2012}, parabolic equations \cite{Cangiani2017}, and poroelasticity problems \cite{Antonietti2019, Botti2016, DeLaPuente2008}. The DG methods are appealing since they guarantee both a high level of precision and flexibility. Moreover, as they can be recast in the context of polygonal and polyhedral grids (PolyDG \cite{Cangiani2014}), they can seamlessly handle complex geometries. 
The analysis of the proposed four-field semi-discrete problem is carried out in the spirit of \cite{Antonietti.Botti.ea:21} and \cite{Botti2021}. We establish a stability estimate under mild requirements on the problem data and tracking the dependencies on the model coefficients and final simulation time. 
We also show that the PolyDG formulation
satisfies a priori $hp$-version error estimates in a suitable energy norms. The theoretical results are supported
by numerical experiments on both benchmark and practical test cases.

The remaining part of the paper is structured as follows: in Section~\ref{sec:continuous_setting} we present the model problem, its four field formulation, and the linearization procedure to deal with the non-linear convective transport term. In Section~\ref{sec:DG_form} we derive the semi-discrete discontinuous Galerkin formulation. Then, in Section~\ref{sec:wp_semidiscr} and Section~\ref{sec:error_est} we derive the stability estimate and a priori $hp$-version error estimates for the semi-discrete problem, respectively. Finally, in Section~\ref{sec:numerical_test} we investigate the convergence performance and the robustness of the proposed method through a complete set of numerical tests with manufactured solutions. Moreover, a simulation inspired by a real case of geothermal energy production is presented.

\section{Model problem and its weak form}
\label{sec:continuous_setting}
Before presenting the differential problem, we introduce the notation for functional spaces.
Given a bounded Lipschitz domain $\omega \in \mathbb{R}^d, d \in \{2;3\}$, we denote by $L^p(\omega)$ the standard Lebesgue spaces on $\omega$ of index $p \in [1, \infty]$ and by $H^m(\omega)$ the Sobolev space of real-valued functions having weak partial derivatives of order up to $m>0$ in $L^2(\omega)$.
For sake of brevity, through the article we often adopt the notation: $(\cdot, \cdot)_{\omega} = (\cdot, \cdot )_{L^2(\omega)}, \ ||\cdot||_{\omega} = ||\cdot||_{L^2(\omega)}$. Finally, for a final time $T_f > 0$ and a Banach space $X$, we denote by $L^p((0,T_f]; X)$ the Bochner space of $X$-valued functions in $L^p((0,T_f])$ endowed with the norm 
\begin{equation}
    ||u||_{L^p((0,T_f];X)} = \left( \int_0^{T_f} ||u(t)||_X^p dt \right)^{\frac{1}{p}}.
\end{equation}

Let $\Omega \subset \mathbb{R}^d$, with $d \in \{2,3\}$, be an open bounded Lipschitz polygonal/polyhedral domain. The thermo-poroelasticity (TPE) problem \cite{Brun2020,Brun2018, Brun2019} reads:
\textit{Find $(\mathbf{u}, p, T)$ such that in $ \Omega \times (0,T_f]$ it holds:} 
\begin{subequations}
    \label{eq:QS_TPE_system}
    \begin{align}
    &\partial_t (a_0 T - b_0 p + \beta \nabla \cdot \mathbf{u}) - c_f \nabla T \cdot (\mathbf{K} \nabla p) - \nabla \cdot (\boldsymbol{\Theta} \nabla T ) = H, \label{eq:energy_cons} \\
    &\partial_t ( c_0 p - b_0 T + \alpha \nabla \cdot \mathbf{u}) - \nabla \cdot (\mathbf{K} \nabla p) = g,  \label{eq:mass_cons} \\
    &-\nabla \cdot \boldsymbol{\sigma}(\mathbf{u},p,T)= \mathbf{f}.  \label{eq:momentum_cons}
    \end{align}
\end{subequations}
Here the variables $\left(\mathbf{u},p,T\right)$ denote the displacement, the fluid pressure and the temperature distribution respectively, while $H, g, \mathbf{f}$ are the source terms, i.e. $H$ is a heat source, $g$ is a fluid mass source, and $\mathbf{f}$ is a body force. We assume that $g,H \in L^2((0,T_f];L^2(\Omega))$ and $\mathbf{f} \in H^1((0,T_f];L^2(\Omega))$. Note that in problem  \eqref{eq:QS_TPE_system}, $T$ represents the variation of the temperature distribution with respect to a reference value \cite{Coussy2003}. Equations \eqref{eq:energy_cons}, \eqref{eq:mass_cons}, and \eqref{eq:momentum_cons} express the conservation of energy, mass and momentum, respectively. We refer the reader to \cite{Brun2018} for the detailed derivation of model \eqref{eq:QS_TPE_system}. The total stress tensor $\boldsymbol{\sigma}$ is expressed in terms of the primary variables $\left(\mathbf{u},p,T\right)$ according to the constitutive law 
\begin{equation}
    \label{eq:const_law_sigma}
    \boldsymbol{\sigma}(\mathbf{u},p,T) = 2 \mu \boldsymbol{\epsilon}(\mathbf{u}) + \lambda \nabla \cdot \mathbf{u} \mathbf{I} - \alpha p \mathbf{I} - \beta T \mathbf{I}, 
\end{equation}
 where $\mathbf{I}$ is the identity tensor and $\boldsymbol{\epsilon}(\mathbf{u})= \frac{1}{2}(\nabla \mathbf{u} + \nabla \mathbf{u}^T)$ is the strain tensor.
Problem \eqref{eq:QS_TPE_system} is endowed with suitable initial conditions $\left(\mathbf{u_0}, p_0, T_0 \right)$, for which we assume the regularity $p_0,T_0 \in H^1_0(\Omega)$ and $\nabla \cdot \mathbf{u}_0 \in L^2(\Omega)$. For simplicity, we close the problem by considering homogeneous Dirichlet boundary conditions. 

\subsection{Thermo-poroelastic coefficients}
\label{subsubsec:relation_coeff}

The physical meaning and the unit of measure of the coefficients characterizing problem \eqref{eq:QS_TPE_system}-\eqref{eq:const_law_sigma} are reported in Table~\ref{tab:TPE_params}.
\begin{table}[H]
    \centering 
    \footnotesize
    \begin{tabular}{c | c  | l}
    \textbf{Notation} & \textbf{Quantity} & \textbf{Unit} \\
    \hline \hline
    $a_0$ & thermal capacity & \si[per-mode = symbol]{\pascal \per \kelvin \squared} \\
    $b_0$ & thermal dilatation coefficient & \si{\per \kelvin} \\
    $c_0$ & specific storage coefficient & \si{\per\pascal} \\
    $\alpha$ & Biot--Willis constant & - \\
    $\beta$ & thermal stress coefficient & \si[per-mode = symbol]{\pascal \per \kelvin} \\
    $c_f$  & fluid volumetric heat capacity divided by reference temperature & \si[per-mode = symbol]{\pascal \per \kelvin\squared} \\ 
    $\mu, \lambda$ & Lamé parameters & \si{\pascal} \\
    $\mathbf{K}$ & permeability divided by fluid viscosity & \si[per-mode = symbol]
    {\metre\squared \per \pascal \per \second} \\
    $\boldsymbol{\Theta}$ & effective thermal conductivity & \si[per-mode = symbol]{\metre\squared \pascal \per \kelvin\squared \per \second} \\
    $K_s$ & matrix bulk modulus & \si{\pascal} \\
    $K_f$ & fluid tangent bulk modulus & \si{\pascal} \\
    $a_f$ & fluid tangent coefficient of volumetric
    thermal dilation & \si{\per \kelvin} \\
    $\phi$ & porosity & - \\
    \end{tabular}
    \\[5pt]
\caption{TPE coefficients appearing in \eqref{eq:QS_TPE_system}, \eqref{eq:const_law_sigma}, and \eqref{eq:relation_coeff}}
\label{tab:TPE_params}
\end{table}
\noindent
Following \cite{Brun2019}, we introduce suitable assumptions on the TPE parameters:
\begin{assumption}[Model coefficients]
    \hspace{0pt}
    \label{assumption:model_problem}
    \begin{enumerate}
    \item the hydraulic mobility $\textup{\textbf{K}}= (K)^d_{i,j=1}$ and heat conductivity $\mathbf{\Theta}=(\Theta)^d_{i,j=1}$ are symmetric tensor fields which, for strictly positive real numbers $k_M>k_m$ and $\theta_M> \theta_m$, satisfy for a.e. $x \in \Omega$ and $\forall\zeta \in \mathbb{R}^d$
    \begin{equation}
        k_m |\zeta|^2 \leq \zeta^T \textup{\textbf{K}}(x)\zeta \leq k_M |\zeta|^2
        \quad \text{and} \quad
        \theta_m |\zeta|^2 \leq \zeta^T \mathbf{\Theta}(x) \zeta \leq \theta_M |\zeta|^2
    \end{equation}
    \item the shear modulus $\mu$ and the fluid heat capacity $c_f$ are scalar fields such that $\mu:\Omega\to[\mu_m,\mu_M]$ and $c_f:\Omega\to[0, c_{fM}]$ with $0<\mu_m\le\mu_M$ and $0\le c_{fM}$;
    \item the constants $\alpha$, $\beta$, $\lambda$ are strictly positive;
    \item the constants $c_0$, $b_0$, $a_0$ are such that $a_0, c_0 \geq b_0\geq 0$.
    \end{enumerate}
\end{assumption}

In what follows, starting from the analysis carried out in \cite{Coussy2003} and taking into account Assumption \ref{assumption:model_problem}, we comment on the relations between the model parameters. First, we point out that the quantity $K=d^{-1}(d \lambda + 2 \mu)$ denotes the bulk modulus of the porous material and that the porosity $\phi\in(0,1)$  represents the percentage of void space in the medium.
The following relations hold:
\begin{equation}
    \label{eq:relation_coeff}
    b_0 = \frac{\beta(\alpha - \phi)}{K}  + \phi a_f, \qquad c_0 = \frac{\alpha - \phi}{K_s} + \frac{\phi}{K_f},\qquad
    \alpha = 1 - \frac{K}{K_s}, 
\end{equation}
where the parameters have the physical meaning summarized in Table~\ref{tab:TPE_params}. Note that from the definition of $\alpha$ we can derive the value of $K_s$ and substitute it into the definition of $c_0$ to get
\[
    K_s = \frac{K}{1 - \alpha} \ \Rightarrow \ c_0 = \frac{(1 - \alpha)(\alpha - \phi)}{K} + \frac{\phi}{K_f}.
\]
Owing to the previous definition of the thermal dilatation coefficient $b_0$ and the storage coefficient $c_0$, it follows from the assumption $c_0\ge b_0$ that
\begin{equation}
    \label{eq:bound_beta}
\beta < 1-\alpha + \gamma_f, \qquad\text{with }\ 
\gamma_f = \frac{K\phi (1-a_f K_f)}{K_f(\alpha-\phi)}.
\end{equation}
Moreover, observing that $K$ and $K_s$ are positive constants satisfying the Hashin--Shtrikman bounds \cite{Hashin1962}, we can obtain a sharper bound for the Biot--Willis coefficient $\alpha$ (see also \cite{Zimmerman2000}), reading
$\phi\le 3 \phi(2 + \phi)^{-1} \leq \alpha \leq 1$.

For the sake of brevity, in what follows we use the symbol $x \lesssim y$ to denote $x < C y$, where $C$ is a positive constant independent of the thermo-poroelastic model's parameters. Without loss of generality and in accordance with the physical interpretation of the model coefficient, we assume that $\gamma_f$ defined in \eqref{eq:bound_beta} satisfies $\gamma_f\lesssim 1$. 

\subsection{Four-field formulation}\label{sec:4field_form}
 
In this section, the four-field formulation of the quasi-static TPE problem \eqref{eq:QS_TPE_system} is presented. The additional scalar equation is meant to ease the stability analysis of the problem without compromising the efficiency of its discretization in terms of computational cost.
Following the idea of \cite{Botti2020, Oyarzua2016}, we introduce the auxiliary variable $\varphi = \lambda \nabla \cdot \mathbf{u} - \alpha p - \beta T$, that represents a volumetric contribution to the total stress. We will refer to this variable as the pseudo-total pressure. 
Then, we can rewrite the divergence of the displacement and the stress tensor as functions of $\varphi$, obtaining the following four-field formulation:
\begin{equation}
    \label{eq:4field_form}
    \left\{
    \begin{aligned}
    &\left(a_0 + \frac{\beta^2}{\lambda}\right)\dot{T} + \left(\frac{\alpha \beta}{\lambda}-b_0\right)\dot{p} + \frac{\beta}{\lambda}\dot{\varphi}  - c_f \nabla T \cdot (\mathbf{K} \nabla p) - \nabla \cdot (\boldsymbol{\Theta} \nabla T ) = H,\\
    & \left(c_0 + \frac{\alpha^2}{\lambda}\right)\dot{p} + \left(\frac{\alpha \beta}{\lambda}-b_0\right)\dot{T} + \frac{\alpha}{\lambda}\dot{\varphi}- \nabla \cdot (\mathbf{K} \nabla p) = g,\\
    &-\nabla \cdot \left(2 \mu \boldsymbol{\epsilon}(\mathbf{u}) + \varphi \mathbf{I}\right) = \mathbf{f},\\
    &\varphi - \lambda \nabla \cdot \mathbf{u} + \alpha p + \beta T = 0 
    \end{aligned}
    \right.
\end{equation}
in $\Omega \times (0,T_f]$. We introduce the functional spaces $\mathbf{V} = \left[H^1_0(\Omega)\right]^d,\ V = H^1_0(\Omega),\ Q = L^2(\Omega)$ and derive in time the last equation in \eqref{eq:4field_form}. Multiplying with appropriate test functions, and summing all the resulting contributions, we obtain the total weak form of \eqref{eq:QS_TPE_system}:
\textit{For any time $t \in (0,T_f]$, find $(\mathbf{u},p,T,\varphi) (t) \in \mathbf{V} \times V \times V \times Q$ such that}
\begin{equation}
    \label{eq:semi_discrete_cont_not_lin}
    \begin{aligned}
    & (b_0 (\dot{p} - \dot{T}),q-S) + ((a_0-b_0)\dot{T},S) + ((c_0-b_0)\dot{p},q) + \frac{1}{\lambda}(\dot{\varphi} + \alpha \dot{p} + \beta \dot{T},\psi + \alpha q + \beta S) \\
    & +
    (\boldsymbol{\Theta} \nabla T,\nabla S) - (c_f \nabla T \cdot (\mathbf{K} \nabla p),S) + (\mathbf{K} \nabla p,\nabla q) +  (2 \mu \boldsymbol{\epsilon}(\mathbf{u}),\boldsymbol{\epsilon}(\mathbf{v}))
    + (\varphi,\nabla \cdot \mathbf{v}) \\
    & - (\nabla \cdot \dot{\mathbf{u}},\psi) = (H,s) + (g,q) + (\mathbf{f},\mathbf{v}) \qquad \forall \ (\mathbf{v},q,S,\psi) \in \mathbf{V} \times V \times V \times Q.
    \end{aligned}
\end{equation}
Problem \eqref{eq:semi_discrete_cont_not_lin} is completed with suitable conditions on the initial pressure field $p(t=0)$, temperature field $T(t=0)$, and divergence of the displacement $\nabla \cdot \mathbf{u}(t=0)$. Indeed, the initial condition on the pseudo-total pressure $\varphi(t=0)$ is inferred from the previous ones according to the fourth equation in \eqref{eq:4field_form}.

\subsection{Linearization and well-posedness}
\label{sec:linearization_wp}

We now propose a linearization procedure to deal with the non-linear convective transport term $\left(c_f \nabla T \cdot (\mathbf{K} \nabla p),S\right)$ in the weak formulation \eqref{eq:semi_discrete_cont_not_lin}. 
Our strategy follows the lines of \cite{Brun2019} where the well-posedness analysis of the thermo-poroelastic non-linear problem is carried out for the fully-mixed variational formulation. 
Denoting by $m \geq 1$ the number of iterations, the fixed-point iterative scheme reads: \textit{for all time $t \in (0,T_f]$, given the temperature gradient at the previous iteration $\nabla T^{m-1}(t)$, find $(\mathbf{u}^m, p^m,T^m,\varphi^m)(t) \in \mathbf{V} \times V \times V \times Q$ such that $ \forall \ (\mathbf{v},q,S,\psi) \in \mathbf{V} \times V \times V \times Q$ it holds:}
\begin{equation}
    \label{eq:nonlin_con_iterative}
    \begin{aligned}
    & b_0(\dot{p}^m - \dot{T}^m,q-S) + ((a_0-b_0)\dot{T}^m,S) + ((c_0-b_0)\dot{p}^m,q) + \frac{1}{\lambda}(\dot{\varphi}^m + \alpha \dot{p}^m + \beta \dot{T}^m,\\
    & \psi + \alpha q + \beta S) +
    (\boldsymbol{\Theta} \nabla T^m,\nabla S) - (c_f \nabla T^{m-1} \cdot (\mathbf{K} \nabla p^m),S) + (\mathbf{K} \nabla p^m,\nabla q) \\
    & +  (2 \mu \boldsymbol{\epsilon}(\mathbf{u}^m),\boldsymbol{\epsilon}(\mathbf{v})) + (\varphi^m,\nabla \cdot \mathbf{v}) - (\nabla \cdot \dot{\mathbf{u}}^m,\psi) = (H,S) + (g,q) + (\mathbf{f},\mathbf{v}),
    \end{aligned}
\end{equation}
together with initial conditions as in \eqref{eq:semi_discrete_cont_not_lin}.
This algorithm must be initialized by an initial guess $\nabla T^0$. Note that, in order to ensure that problem \eqref{eq:nonlin_con_iterative} is well-defined, we have to introduce a regularity hypothesis:
\begin{assumption}
    \label{ass:heat_flux}
    We suppose that $c_f \nabla T^m \in (L^{\infty}(\Omega))^d$ for all $m \geq 0, \ t \in (0,T_f]$. 
\end{assumption}
\begin{remark}
As pointed out in \cite{Brun2019}, the fixed-point iterative procedure and its convergence analysis can also be conducted by linearizing the Darcy flux. In this case we would have \\ $(c_f \nabla T^m \cdot (\mathbf{K} \nabla p^{m-1}),S)$ in \eqref{eq:nonlin_con_iterative} and Assumption~\ref{ass:heat_flux} would be on the regularity of $\nabla p^m$. We additionally observe that Assumption~\ref{ass:heat_flux} can be weakened to $c_f\nabla T^m \in (L^{3}(\Omega))^d$ by using the generalized H\"older and Poincar\'e--Sobolev inequalities to infer the continuity of the convective transport term, i.e.
\begin{equation}
    \begin{aligned}
    (c_f \nabla T^{m-1} \cdot (\mathbf{K} \nabla p^m),S)
    & \ \le ||c_f \nabla T^{m-1}||_{L^{3}(\Omega)^d} 
    ||\mathbf{K} \nabla p^m||_{L^{2}(\Omega)^d}||S||_{L^{6}(\Omega)} \\
    & \ \lesssim k_M ||c_f\nabla T^{m-1}||_{L^{3}(\Omega)^d} ||p^m||_V ||S||_V. 
    \end{aligned}
\end{equation}
However, we prefer to adopt Assumption~\ref{ass:heat_flux} in order to avoid additional technicalities in the numerical analysis of the proposed discretization method. 
\end{remark}

The convergence of the fixed-point iterative scheme \eqref{eq:nonlin_con_iterative} can be obtained by adapting the argument of \cite{Brun2019} (see also \cite{VanDuijn2004} where a similar technique is used for crystal dissolution and precipitation in porous media). First, the well-posedness of a linearized version of the weak formulation \eqref{eq:semi_discrete_cont_not_lin} corresponding to one iteration of \eqref{eq:nonlin_con_iterative} is established. 
The linearized variational formulation reads:
\textit{For any $t \in (0,T_f]$, find $(\mathbf{u},p,T,\varphi)(t) \in \mathbf{V} \times V \times V \times Q$ such that:}
\begin{equation}
    \begin{aligned}
    & b_0(\dot{p} - \dot{T},q-S) + ((a_0-b_0)\dot{T},S) + ((c_0-b_0)\dot{p},q) + \frac{1}{\lambda}(\dot{\varphi} + \alpha \dot{p} + \beta \dot{T},\psi + \alpha q + \beta S) \\
    & + (\boldsymbol{\Theta} \nabla T,\nabla S) - ( \mathbf{K} \nabla p, \boldsymbol{\eta}S) + (\mathbf{K} \nabla p,\nabla q) +  (2 \mu \boldsymbol{\epsilon}(\mathbf{u}),\boldsymbol{\epsilon}(\mathbf{v}))
    + (\varphi,\nabla \cdot \mathbf{v}) - (\nabla \cdot \dot{\mathbf{u}},\psi) \\
    & = (H,s) + (g,q) + (\mathbf{f},\mathbf{v}) \qquad \forall \ (\mathbf{v},q,S,\psi) \in \mathbf{V} \times V \times V \times Q,\\
    \end{aligned}
    \label{eq:semi_discrete_cont}
\end{equation}
where, for some given $\boldsymbol{\eta} \in \left( L^{\infty}(\Omega) \right)^d$, the term $\left(\mathbf{K} \nabla p, \boldsymbol{\eta} S\right)$ replaces the non-linear transport term $\left(c_f \nabla T \cdot (\mathbf{K} \nabla p),S\right)$. Then, the convergence of the iterative procedure \eqref{eq:semi_discrete_cont_not_lin} to the weak solution of the non-linear problem follows by applying the Banach fixed-point Theorem \cite{Evans1998} and the results obtained for the linearized problem \eqref{eq:semi_discrete_cont}.

In what follows, we construct approximate PolyDG solutions to problem \eqref{eq:semi_discrete_cont}, for which we derive suitable a priori estimates. We apply the fixed-point iterative scheme \eqref{eq:nonlin_con_iterative} to the PolyDG discretization of the TPE non-linear problem and assess numerically its convergence performance. The theoretical analysis of the linearization method and the proof of the convergence to the weak solution of problem \eqref{eq:semi_discrete_cont_not_lin} will be the focus of a future work.

\section{Discontinuous Galerkin semi-discrete problem}
\label{sec:DG_form}
To derive the semi-discrete PolyDG approximation of the TPE problem we introduce a polytopic subdivision $\mathcal{T}_h$ of the computational domain $\Omega$. An interface is defined as a planar subset of the intersection of the boundaries of any two neighbouring elements of $\mathcal{T}_h$. We remark that if $d = 2$ an interface is a line segment, while if $d = 3$ an interface is a planar polygon, that can be further decomposed into a set of triangles. We denote with $\mathcal{F}_B$ and $\mathcal{F}_I$ the set of boundary and interior faces, respectively, and we set $\mathcal{F}=\mathcal{F}_B\cup\mathcal{F}_I$.
In what follows, we introduce the main assumptions on the mesh $\mathcal{T}_h$ (cf. \cite{CangianiDong:17, Cangiani2014}).
\begin{definition}[Polytopic regular mesh]
\label{def:unif_regular}
A mesh $\mathcal{T}_h$ is polytopic regular if for any $ \kappa \in \mathcal{T}_h$, there exist a set of non-overlapping simplices contained in $\kappa$, denoted by $\{S_{\kappa}^F\}_{F \subset \partial \kappa}$, such that, for any face $F \subset \partial \kappa$, the following condition holds: $$
h_{\kappa} \lesssim d \ |S_{\kappa}^F| \ |F|^{-1},
$$
with $h_{\kappa}$ denoting the diameter of the element $\kappa$.
\end{definition}

\begin{assumption}
\label{ass:mesh_Th1}
The mesh sequence $\{\mathcal{T}_h\}_h$ satisfies the following properties:
\begin{enumerate}[start=1,label={\bfseries A.\arabic*}]
    \item \label{ass:A1} $\{\mathcal{T}_h\}_h$ is uniformly polytopic-regular;
    \item \label{ass:A2} For each $\mathcal{T}_h\in \{\mathcal{T}_h\}_h$ there exists a shape-regular, simplicial covering $\mathcal{T}_h^*$ of $\mathcal{T}_h$ such that, for each pair $\kappa \in \mathcal{T}_h$ and $k \in \mathcal{T}_h^*$ with $\kappa \subset k$ it holds
    \begin{enumerate}[start=1,label={(\roman*)}]
    \item $h_{k} \lesssim h_{\kappa}$;
    \item $\underset{\kappa \in \mathcal{T}_h}{\mbox{max}} \ \mbox{card} \left\{\kappa' \in \mathcal{T}_h: \kappa' \cap k \neq 0, k \in \mathcal{T}_h^*, \kappa \subset k \right\} \lesssim 1$;
    \end{enumerate}
\end{enumerate}
\end{assumption}
We remark that, under \ref{ass:A1} the following inequality holds \cite{Cangiani2017}:
\begin{equation}
    \label{eq:trace_inverse_ineq}
    ||v||_{L^2(\partial \kappa)} \lesssim \frac{\ell}{h_{\kappa}^{1/2}} ||v||_{L^2(\kappa)} \quad \forall v \in \mathbb{P}^{\ell}(\kappa),
\end{equation}
where $\mathbb{P}^{\ell}(k)$ is the space of polynomials of degree less than or equal to $\ell$ in $\kappa$ and the hidden constant is independent of $\ell, h_{\kappa}$, and of the number of faces per element. We refer to \eqref{eq:trace_inverse_ineq} as discrete trace-inverse inequality.
Then, we introduce the average and jump operators on each interior face $F\in\mathcal{F}_I$ shared by the elements $\kappa^{\pm}$ as in \cite{Arnold2002}:
\begin{equation}
    \label{eq:avg_jump_operators}
    \begin{aligned}
    & \jump{a} = a^+ \mathbf{n^+} + a^- \mathbf{n^-}, \ 
    && \jump{\mathbf{a}} = \mathbf{a}^+ \odot \mathbf{n^+} + \mathbf{a}^- \odot \mathbf{n^-}, \ 
    &&\jump{\mathbf{a}}_n = \mathbf{a}^+ \cdot \mathbf{n^+} + \mathbf{a}^- \cdot \mathbf{n^-}, \\ 
    & \avg{a} = \frac{a^+ + a^-}{2}, \
    && \avg{\mathbf{a}} = \frac{\mathbf{a}^+ + \mathbf{a}^-}{2}, \ && \avg{\mathbf{A}} = \frac{\mathbf{A}^+ + \mathbf{A}^-}{2},
    \end{aligned}
\end{equation}
where $\mathbf{a} \odot \mathbf{n} = \mathbf{a}\mathbf{n}^T$, and $a, \ \mathbf{a}, \ \mathbf{A}$ are scalar-, vector-, and tensor-valued functions, respectively. The notation $(\cdot)^{\pm}$ is used for the trace on $F$ taken within $\kappa^\pm$ and $\mathbf{n}^\pm$ is the outer normal vector to $\partial \kappa^\pm$. Accordingly, on boundary faces $F\in\mathcal{F}_B$, we set
$
 \jump{a} = a \mathbf{n},\
\avg{a} = a,\
\jump{\mathbf{a}} = \mathbf{a} \odot \mathbf{n},\
\avg{\mathbf{a}} = \mathbf{a},\\
\jump{\mathbf{a}}_n = \mathbf{a} \cdot \mathbf{n},\
\avg{\mathbf{A}} = \mathbf{A}.$ 
For the sake of simplicity, we assume that the parameters $\boldsymbol{\Theta}, \mathbf{K}$, and $\mu$ are element-wise constant. Then, we can introduce the quantities
$
    \overline{\Theta}_{\kappa} = |\sqrt{\boldsymbol{\Theta}\rvert_{\kappa}}|_2^2, 
    \, \overline{K}_{\kappa} = |\sqrt{\mathbf{K}\rvert_{\kappa}}|_2^2, 
    \,\text{and } \\ \mu_{\kappa} = \mu \rvert_{\kappa},
$
where $|\cdot|_2$ denotes the $\ell^2$-norm in $\mathbb{R}^{d \times d}$. 

We now proceed deriving the semi-discrete formulation of problem \eqref{eq:4field_form}. First, for $m,\ell \geq 1$ we introduce the discrete spaces
\begin{equation}
    \begin{aligned}
    Q_h^{m} &= \left\{ v_h \in L^2(\Omega) : v_h |_{\kappa} \in \mathbb{P}^{m}(\kappa) \ \ \forall \kappa \in \mathcal{T}_h \right\},\\
    V_h^{\ell} &= \left\{ v_h \in L^2(\Omega) : v_h |_{\kappa} \in \mathbb{P}^{\ell}(\kappa) \ \ \forall \kappa \in \mathcal{T}_h \right\},\\
    \mathbf{V}_h^{\ell} &= \left\{ \mathbf{v}_h \in \mathbf{L}^2(\Omega) : \mathbf{v}_h |_{\kappa} \in     \left[\mathbb{P}^{\ell}(\kappa)\right]^d \ \ \forall \kappa \in \mathcal{T}_h \right\}.\\
    \end{aligned}
\end{equation}
In the following discussion, we focus on a Symmetric Interior Penalty formulation \cite{Arnold1982, Ern2021, Wheeler1978}. Thus, the PolyDG semi-discretization of problem \eqref{eq:semi_discrete_cont_not_lin} reads:
\textit{For any $t \in (0,T_f]$, find $(\mathbf{u}_h,p_h,T_h,\varphi_h)(t) \\ \in \mathbf{V}_h^{\ell} \times V_h^{\ell} \times V_h^{\ell} \times Q_h^m$ such that}
$$
    \begin{aligned}
    & b_0(\dot{p}_h - \dot{T}_h,q_h-S_h) + ((a_0-b_0)\dot{T}_h,S_h) + ((c_0-b_0)\dot{p}_h,q_h) + \frac{1}{\lambda}(\dot{\varphi}_h + \alpha \dot{p}_h + \beta \dot{T}_h,\\
    & \psi_h + \alpha q_h + \beta S_h)
    + \mathcal{A}_h^{T}(T_h,S_h) + \widetilde{\mathcal{C}}_h(T_h,p_h,S_h)
    + \mathcal{A}_h^{p}(p_h,q_h) + \mathcal{A}_h^{e}(\mathbf{u}_h,\mathbf{v}_h) \\
    & - \mathcal{B}_h(\varphi_h,\mathbf{v}_h) +  \mathcal{B}_h(\psi_h,\dot{\mathbf{u}}_h) +  \mathcal{D}_h(\dot{\varphi_h},\psi_h)
    = (H,S_h) + (g,q_h) + (\mathbf{f},\mathbf{v}_h) 
    \end{aligned}
$$
$\forall (\mathbf{v}_h,q_h,S_h,\varphi_h) \in \mathbf{V}_h^{\ell} \times V_h^{\ell} \times V_h^{\ell} \times Q_h^m$, with initial conditions $(\mathbf{u}_{h,0},p_{h,0},T_{h,0},\varphi_{h,0})$ that are suitable approximation of the initial data of the model problem \eqref{eq:QS_TPE_system} and bilinear/trilinear forms defined by
\begin{equation}
    \label{eq:bilinear_forms_discr}
    \begin{aligned}
    & \begin{aligned}
    \mathcal{A}_h^T(T,S) = &\left(\boldsymbol{\Theta}\nabla_h T, \nabla_h S\right) -  \hspace{-0.2cm} \sum_{F \in \mathcal{F}} \int_F  \hspace{-0.2cm} \left(\avg{\boldsymbol{\Theta}\nabla_h T} \mkern-2.5mu \cdot \mkern-2.5mu \jump{S} + \jump{T} \mkern-2.5mu \cdot \mkern-2.5mu \avg{\boldsymbol{\Theta}\nabla_h S} + \sigma \jump{T} \mkern-2.5mu \cdot \mkern-2.5mu \jump{S} \right),
    \end{aligned}\\
    & \begin{aligned}
    \mathcal{A}_h^p(p,q) = (\mathbf{K} \nabla_h p,\nabla_h q) - \hspace{-0.2cm} \sum_{F \in \mathcal{F}} \int_F \hspace{-0.2cm} \left(\avg{\mathbf{K} \nabla_h p} \mkern-2.5mu \cdot \mkern-2.5mu \jump{q} + \jump{p} \mkern-2.5mu \cdot \mkern-2.5mu \avg{\mathbf{K} \nabla_h q} + \xi \jump{p} \mkern-2.5mu \cdot \mkern-2.5mu \jump{q} \right),
    \end{aligned}\\
    & \begin{aligned}
    \mathcal{A}_h^e(\mathbf{u},\mathbf{v}) = & (2 \mu\boldsymbol{\epsilon}_h(\mathbf{u}),\boldsymbol{\epsilon}_h(\mathbf{v})) -  \sum_{F \in \mathcal{F}} \int_F \bigg( \avg{2 \mu\boldsymbol{\epsilon}_h(\mathbf{u})} \mkern-2.5mu : \mkern-2.5mu \jump{\mathbf{v}} + \jump{\mathbf{u}} \mkern-2.5mu : \mkern-2.5mu \avg{2 \mu\boldsymbol{\epsilon}_h(\mathbf{v})} \\
    & + \zeta \jump{\mathbf{u}} \mkern-2.5mu : \mkern-2.5mu \jump{\mathbf{v}} \bigg),
    \end{aligned} \\
    & \mathcal{B}_h(\varphi,\mathbf{v}) = - (\varphi,\nabla_h \cdot \mathbf{v}) + \sum_{F \in \mathcal{F}} \int_F \avg{ \varphi} \mkern-2.5mu \cdot \mkern-2.5mu \jump{\mathbf{v}}_n,\\
    & \widetilde{\mathcal{C}}_h (T,p,S) = - (c_f \nabla_h T \cdot \left(\mathbf{K} \ \nabla_h p\right),S),\\
    & \mathcal{D}_h(\varphi,\psi) = \sum_{F \in \mathcal{F}_I} \int_F \varrho \jump{\varphi} \mkern-2.5mu \cdot \mkern-2.5mu \jump{\psi}.\\
    \end{aligned}
\end{equation}
Here, for all $w\in V_h^{\ell}$ and $\mathbf{w}\in \mathbf{V}_h^{\ell}$, $\nabla_h w$ and $\nabla_h \cdot \mathbf{w}$ denote the broken differential operators whose restrictions to each element $k \in \mathcal{T}_h$ are defined as $\nabla w_{|k}$ and $\nabla \cdot w_{|k}$, respectively, and $\boldsymbol{\epsilon}_h(\mathbf{u}) = \left(\nabla_h \mathbf{u} + \nabla_h \mathbf{u}^T\right)/2$. The stabilization functions $\sigma, \xi, \zeta, \varrho \in L^{\infty}(\mathcal{F}_h)$ are defined according to \cite{Cangiani2017}:
\begin{equation}
    \label{eq:stabilization_func}
    \begin{aligned}
    \sigma &= \left\{\begin{aligned}
    &\alpha_1 \underset{\kappa \in \{\kappa^+,\kappa^-\}}{\mbox{max}}\left( \frac{\overline{\Theta}_{\kappa} \ell^2}{h_{\kappa}}\right) \ & F \in \mathcal{F}_I,\\
    &\alpha_1 \overline{\Theta}_{\kappa} \ell^2 h_{\kappa}^{-1} \ & F \in \mathcal{F}_B,\\
    \end{aligned}
    \right.
    \ \
     \xi = &&\left\{\begin{aligned}
    &\alpha_2 \underset{\kappa \in \{\kappa^+,\kappa^-\}}{\mbox{max}}\left( \frac{\overline{K}_{\kappa} \ell^2}{h_{\kappa}}\right) \ & F \in \mathcal{F}_I,\\
    &\alpha_2 \overline{K}_{\kappa} \ell^2 h_{\kappa}^{-1} \ & F \in \mathcal{F}_B,\\
    \end{aligned}
    \right.\\
    \zeta &= \left\{\begin{aligned}
    &\alpha_3 \underset{\kappa \in \{\kappa^+,\kappa^-\}}{\mbox{max}}\left( \frac{\mu_{\kappa} \ell^2}{h_{\kappa}}\right) \ & F \in \mathcal{F}_I,\\
    &\alpha_3 \mu_{\kappa} \ell^2 h_{\kappa}^{-1} \ & F \in \mathcal{F}_B,\\
    \end{aligned}
    \right.
    \ \
    \varrho = &&\left\{\begin{aligned}
    &\alpha_4 \underset{\kappa \in \{\kappa^+,\kappa^-\}}{\mbox{min}}\left(\frac{h_{\kappa}}{m}\right) \ & F \in \mathcal{F}_I,\\
    &\alpha_4 h_{\kappa} m^{-1} \ & F \in \mathcal{F}_B,\\
    \end{aligned}
    \right.\\
    \end{aligned}   
\end{equation}
where $\alpha_1, \alpha_2, \alpha_3, \alpha_4 \in \mathbb{R}$ are positive constants to be properly defined. We point out that in the formulation above, we have decided to consider the same polynomial degree for the spaces $V_h^{\ell}$ and $\mathbf{V}_h^{\ell}$, because we are mainly interested in approximation schemes yielding the same accuracy for the pore pressure, temperature, and displacement.

Finally, as done for the continuous case in Section~\ref{sec:linearization_wp}, we introduce the linearized version of the PolyDG formulation consisting in replacing the non-linear convective term $\widetilde{\mathcal{C}}_h(T_h,p_h,S_h)$ by the bilinear form $$\mathcal{C}_h(p_h, S_h) = -(\mathbf{K}\ \nabla_h p_h, \ \boldsymbol{\eta} S_h),$$ 
defined for a given vector field $\boldsymbol{\eta} \in \left( L^{\infty}(\Omega) \right)^d$. 
In order to ease the notation, we introduce $X_h = (\mathbf{u}_h,p_h,T_h,\varphi_h), Y_h = (\mathbf{v}_h,q_h,S_h,\psi_h) \in \mathbf{X}_h = \mathbf{V}_h^{\ell} \times V_h^{\ell} \times V_h^{\ell} \times Q_h^m$ and we write the linearized semi-discrete variational formulation as:
\textit{For any time $t \in (0,T_f]$, find $X_h \in\mathbf{X}_h$ such that:}
\begin{equation}
    \label{eq:discrete_weak_form2}
    \mathcal{M}_h(\dot{X}_h,Y_h) + \mathcal{A}_h(X_h,Y_h) - \mathcal{B}_h(\varphi_h,\mathbf{v}_h) + \mathcal{B}_h(\psi_h,\dot{\mathbf{u}}_h) = F(Y_h) \qquad \forall \ Y_h \in \mathbf{X}_h,
\end{equation}
where the bilinear form $\mathcal{M}_h, \mathcal{A}_h :\mathbf{X}_h\times\mathbf{X}_h\to\mathbb{R}$ are defined such that
\begin{equation}
    \label{eq:LHS_wp}
    \begin{aligned}
    \mathcal{M}_{h}(X_h,Y_h) &=  b_0(p_h - T_h,q_h-S_h) + ((a_0-b_0)T_h,S_h) + ((c_0-b_0)p_h,q_h)  \\ 
    &+ \lambda^{-1}(\varphi_h + \alpha p_h + \beta T_h, \psi_h + \alpha q_h + \beta S_h) + \mathcal{D}_h(\varphi_h,\psi_h), \\
    \mathcal{A}_{h}(X_h,Y_h) &= \mathcal{A}_h^T(T_h,S_h) + \mathcal{C}_h(p_h,S_h) + \mathcal{A}_h^p(p_h,q_h) + \mathcal{A}_h^e (\mathbf{u}_h,\mathbf{v}_h);
    \end{aligned}
\end{equation}
and the expression of the linear functional in the right-hand side of \eqref{eq:discrete_weak_form2} is given by \\ $F(Y_h)= (H,S_h) + (g,q_h) + (\mathbf{f},\mathbf{v}_h).$
\begin{remark}
Note that, following the DG discretization of the Stokes problem analyzed in \cite{antonietti2020_stokesDG}, in the semi-discrete formulation \eqref{eq:discrete_weak_form2}-\eqref{eq:LHS_wp} we have added an additional weakly consistent stabilization term for the pseudo-total pressure.
\end{remark}

\section{Stability analysis}
\label{sec:wp_semidiscr}
The aim of this Section and Section~\ref{sec:error_est} is to perform a complete numerical analysis of the linearized PolyDG semi-discretization \eqref{eq:discrete_weak_form2}. Before establishing an a priori estimate we define the energy norms and present some preliminary results.
For carrying out our analysis we first define, for an integer $l \geq 1$, the broken Sobolev spaces
\begin{equation}
   \begin{aligned}
   H^{l}(\mathcal{T}_h) & \ = \left\{ v_h \in L^2(\Omega) : v_h |_{\kappa} \in H^{l}(\kappa) \ \ \forall \kappa \in \mathcal{T}_h \right\}, \\
   \mathbf{H}^{l}(\mathcal{T}_h) & \ = \left\{ \mathbf{v}_h \in \mathbf{L}^2(\Omega) : \mathbf{v}_h |_{\kappa} \in \mathbf{H}^{l}(\kappa) \ \ \forall \kappa \in \mathcal{T}_h \right\}, 
   \end{aligned}
\end{equation}
and we introduce the shorthand notation $||\cdot||=||\cdot||_\Omega$ and $||\cdot||_{\mathcal{F}}=\left(\sum_{F\in\mathcal{F}}||\cdot||_F^2\right)^{\frac12}$.
The $DG$-norms that will be used in the analysis are defined such that
\begin{equation}
    \label{eq:DG_norms}
    \begin{aligned}
    &||S||^2_{DG,T} = ||\sqrt{\boldsymbol{\Theta}} \ \nabla_h S||^2 + ||\sqrt{\sigma} \jump{S}||_{\mathcal{F}}^2 \ &&  \forall \ S \in V_h^{\ell},\\ 
    &||q||^2_{DG,p} = ||\sqrt{\mathbf{K}} \ \nabla_h q||^2 + ||\sqrt{\xi} \jump{q}||_{\mathcal{F}}^2  \quad && \forall \ q \in V_h^{\ell},\\ 
    &||\mathbf{v}||^2_{DG,e} = ||\sqrt{2 \mu} \ \boldsymbol{\epsilon}_h(\mathbf{v})||^2 + ||\sqrt{\zeta} \jump{\mathbf{v}}||_{\mathcal{F}}^2 \ && \forall \ \mathbf{v} \in \mathbf{V}_h^{\ell}.
    \end{aligned}
\end{equation}
We can now state the boundedness and coercivity of the bilinear forms in \eqref{eq:bilinear_forms_discr}. Since the proof hinges on standard arguments on DG discretizations, we  refer the reader to \cite[Section 3]{AntoniettiMazzieri2018} for all the details. 
\begin{lemma}
\label{lem:boundcoerc_bil_forms}
Let Assumption~\ref{assumption:model_problem} and    Assumption~\ref{ass:mesh_Th1} be satisfied. Assume that the parameters $\alpha_1$, $\alpha_2$, and $\alpha_3$ appearing in \eqref{eq:stabilization_func} are chosen sufficiently large. Then
\begin{equation}
    \begin{aligned}
    \mathcal{A}_h^T(T,S) \lesssim \ & ||T||_{DG,T} ||S||_{DG,T}, \ && \mathcal{A}_h^T(T,T) \gtrsim ||T||_{DG,T}^2 \ &&\forall \ T,S \in V_h^{\ell},\\
    \mathcal{A}_h^p(p,q) \lesssim \ & ||p||_{DG,p} ||q||_{DG,p}, \ && \mathcal{A}_h^p(p,p) \gtrsim ||p||_{DG,p}^2 \ &&\forall \ p,q \in V_h^{\ell},\\
    \mathcal{A}_h^e(\mathbf{u},\mathbf{v}) \lesssim \ & ||\mathbf{u}||_{DG,e} ||\mathbf{v}||_{DG,e}, \ && \mathcal{A}_h^e(\mathbf{u},\mathbf{u}) \gtrsim ||\mathbf{u}||_{DG,e}^2 \ &&\forall \ \mathbf{u},\mathbf{v} \in \mathbf{V}_h^{\ell},\\
    \end{aligned}
\end{equation}
\end{lemma}

The next Proposition establishes the positivity of the bilinear forms $\mathcal{M}_h$ and $\mathcal{A}_h$ in \eqref{eq:discrete_weak_form2} and the inf-sup stability of the hydro-mechanical coupling given by $\mathcal{B}_h$.
\begin{proposition}
  \label{prop:pos}
Let Assumption~\ref{assumption:model_problem} and Assumption~\ref{ass:mesh_Th1} hold and assume that the parameters $\alpha_1$, $\alpha_2$, $\alpha_3$, and $\alpha_4$ appearing in \eqref{eq:stabilization_func} are large enough. Then
\begin{enumerate}
    \item[(i)] for all $Y_h = (\mathbf{v}_h, q_h, S_h, \psi_h)\in \mathbf{X}_h$ it holds
    \begin{equation}
        \label{eq:positivity_Mh}
        \mathcal{M}_h(Y_h, Y_h) \gtrsim 
        ||\sqrt{a_0-b_0}\, S_h ||^2 + ||\sqrt{c_0-b_0}\, q_h ||^2 + ||\sqrt{d_0}\, \psi_h ||^2
        +\mathcal{D}_h(\psi_h, \psi_h),
    \end{equation}
    with $d_0 = (1+\gamma_f-\alpha-\beta)(\alpha-\phi)K^{-1}$ and $\gamma_f$ defined in \eqref{eq:bound_beta};
    \vspace{1mm}
    \item[(ii)] under the additional requirement $||\boldsymbol{\eta}||_{L^\infty(\Omega)^d}\lesssim \sqrt{ \theta_m k_M^{-1}}$,  for all $Y_h\in \mathbf{X}_h$ 
    \begin{equation}
        \label{eq:positivity_Ah}
        \mathcal{A}_h(Y_h, Y_h) \gtrsim 
        ||S_h||^2_{DG,T} + ||q_h||^2_{DG,p}+||\mathbf{v}_h||^2_{DG,e};
    \end{equation}
    \item[(iii)] assuming that the polynomial degrees $\ell$ and $m$  satisfy $\ell+1 \geq m$, the bound 
\begin{equation}
    \label{eq:gen_inf_sup}
    \underset{\mathbf{0} \neq \mathbf{v}_h \in \mathbf{V}^{\ell}_h}{\mbox{sup}} \frac{\mathcal{B}_h(\mathbf{v}_h, \varphi_h)}{||\mathbf{v}_h||_{DG,e}} + \mathcal{D}_h(\varphi_h,\varphi_h)^{\frac12} \geq \mathbb{B} ||\varphi_h|| \qquad \forall \varphi_h \in Q_h^m
\end{equation}
is valid with $\mathbb{B}>0$ depending on $\ell$ and $m$ but independent of the mesh size $h$.
\end{enumerate}
\end{proposition}
\begin{proof} \textbf{(\textit{i})} Let $Y_h=(\mathbf{v}_h, q_h, S_h, \psi_h)\in\mathbf{X}_h$. In order to prove \eqref{eq:positivity_Mh}, we recall the definition of the thermal dilatation coefficient $b_0$ and bulk modulus $K$ to infer 
\begin{equation}
\label{eq:bnd_Mh}
\begin{aligned}
\mathcal{M}_h(Y_h, Y_h) & \ge ||\lambda^{-\frac12} (\psi_h + \alpha q_h + \beta S_h)||^2 +
||\sqrt{a_0-b_0}\, S_h ||^2  \\ &\quad +  
||\sqrt{c_0-b_0}\, q_h ||^2 + ||\sqrt{b_0}\, (S_h-q_h)||^2
+\mathcal{D}_h(\psi_h, \psi_h)\\
&\ge ||K^{-\frac12} (\psi_h + \alpha q_h + \beta S_h)||^2 + 
||\sqrt{a_0-b_0}\, S_h ||^2+\mathcal{D}_h(\psi_h, \psi_h) \\ &\quad + ||\sqrt{c_0-b_0}\, q_h ||^2 +||\sqrt{\beta(\alpha-\phi)K^{-1}} (S_h-q_h)||^2.
\end{aligned}
\end{equation}
Then, we let $d_0 = (1+\gamma_f-\alpha-\beta) (\alpha-\phi)K^{-1}\ge 0$ and we express $\psi_h$ as a linear combination of the terms appearing in the right-hand side of the previous bound, i.e. 
$$
\sqrt{d_0} \psi_h = c_1 \left(\frac{\psi_h+\alpha q_h+\beta S_h}{\sqrt{K}}\right) + c_2 \frac{\sqrt{\beta(\alpha-\phi)}}{\sqrt{K}}\, (S_h-q_h) + c_3 \sqrt{c_0-b_0}\ q_h,
$$
with coefficients $c_1 = \sqrt{(1+\gamma_f-\alpha-\beta)(\alpha-\phi)}\lesssim 1$, $c_2 =\sqrt{\beta(1+\gamma_f-\alpha-\beta)}\lesssim 1$ and \\ $c_3 = \alpha + \beta \lesssim 1$ according to \eqref{eq:bound_beta}.
Hence, using the triangle inequality one has 
$$
||\sqrt{d_0}\, \psi_h ||^2 \lesssim 
\left|\left|\sqrt{\frac{c_0-b_0}2}\, q_h \right|\right|^2 +\left|\left|\frac{\sqrt{\beta(\alpha-\phi)}}{\sqrt{K}} (S_h-q_h)\right|\right|^2 + \left|\left|\frac{\psi_h + \alpha q_h + \beta S_h}{\sqrt{K}}\right|\right|^2.
$$
Finally, plugging the previous bound into \eqref{eq:bnd_Mh} yields the conclusion.

\textbf{(\textit{ii})} We now proceed with the proof of  property \eqref{eq:positivity_Ah}. First, we observe that we can bound the transport term $\mathcal{C}_h$ from below by using a discrete Poincar\'e inequality (cf. \cite{Brenner:03} and \cite[Corollary 5.4]{DiPietro2012}) together with the Young inequality to obtain 
\begin{equation}
    \label{eq:wp_Bh}
    \begin{aligned}
    \mathcal{C}_h(q_h,S_h) \gtrsim - \frac12  ||\sqrt{\mathbf{K}} \nabla_h p_h||^2 - \frac{k_M ||\boldsymbol{\eta}||_{L^\infty(\Omega)^d}^2}{2\theta_m} ||\sqrt{\boldsymbol{\Theta}} \ \nabla_h T_h||^2.
    \end{aligned}
\end{equation}
Thus, owing to the coercivity properties stated in Lemma \ref{lem:boundcoerc_bil_forms} it is inferred that
\begin{equation}
    \label{eq:wp_Iterm}
    \begin{aligned}
    & \mathcal{A}_{h}(Y_h,Y_h) \gtrsim 
    \left(1 - \frac{k_M ||\boldsymbol{\eta}||_{L^\infty(\Omega)^d}^2}{2\theta_m}\right)||S_h||_{DG,T}^2 + ||q_h||_{DG,p}^2 + ||\mathbf{v}_h||_{DG,e}^2,
    \end{aligned}
\end{equation}
which corresponds to \eqref{eq:positivity_Ah} under the assumption that $\theta_m \gtrsim k_M ||\boldsymbol{\eta}||_{L^\infty(\Omega)^d}^2$.

\textbf{(\textit{iii})} 
The proof of condition \eqref{eq:gen_inf_sup} follows from \cite[Proposition 3.1]{antonietti2020_stokesDG}, which hinges on the inverse trace inequality \eqref{eq:trace_inverse_ineq} and the fact that $\nabla_h \varphi_h \in \mathbf{V}_h^{\ell}$ for all $\varphi_h \in Q_h^m$. 
\end{proof}
\begin{remark}
The theoretical requirement on $||\boldsymbol{\eta}||_{L^\infty(\Omega)^d}$ introduced to prove \eqref{eq:positivity_Ah} is meant to simplify the stability analysis but, as we observe in the robustness test cases of Section \ref{sec:robustness_analysis}, is not needed in practice. Indeed, it is possible to weaken the assumption by controlling the convective term $\mathcal{C}_h$ using the generalized inf-sup condition \eqref{eq:gen_inf_sup} and the positive terms in \eqref{eq:positivity_Mh}. 
\end{remark}

\subsection{Stability estimates}
\label{subsec:stability_est}
In this section we derive the a priori estimate for the semi-discrete problem \eqref{eq:discrete_weak_form2}. To ease the notation we define the norm 
\begin{equation}
    ||X_h||_{\mathcal{E}}^2 = ||(\mathbb{B}+d_0)^\frac12 \varphi_h||^2 +(a_0-b_0)||T_h||^2 + (c_0-b_0)||p_h||^2 + ||\mathbf{u}_h||_{DG,e}^2,
\end{equation}
for all $X_h = (\mathbf{u}_h, p_h, T_h, \varphi_h)\in \mathbf{X}_h$, with $d_0$ and $\mathbb{B}$ defined as in Proposition \ref{prop:pos}. We remark that the stability estimate below shows a linear dependence on the final time $T_f$ since we are able to establish the result without resorting to the Gr\"onwall Lemma.
\begin{theorem}
\label{thm:stability_est}
Let the assumptions of Proposition \ref{prop:pos} be satisfied and let $X_h = (\mathbf{u}_h, p_h, T_h, \varphi_h)(t) \\ \in \mathbf{X}_h$ be the solution of \eqref{eq:discrete_weak_form2} for any $t \in (0,T_f]$. Then, it holds
\begin{equation}
\begin{aligned}
    & \sup_{t\in (0,T_f]} \left(||X_h||_{\mathcal{E}}^2\right) + 
    \int_0^{T_f}\left(
    ||T_h(s)||_{DG,T}^2+||p_h(s)||_{DG,p}^2\right){\rm d}s
    \lesssim \mathcal{R}_0 \,+
    \\&\qquad 
    \theta_m^{-1}||H||_{L^2((0,T_f];L^2(\Omega))}^2 
    + k_m^{-1}||g||_{L^2((0,T_f];L^2(\Omega))}^2
    + (T_f+\mu_m^{-1})||\mathbf{f}||_{H^1((0,T_f];L^2(\Omega))}^2,
\end{aligned}
\end{equation}
with $\mathcal{R}_0 = \mathcal{M}_h(X_{h,0}, X_{h,0}) +||\mathbf{u}_{h,0}||_{DG,e}^2 +(1+\mu_m^{-1})||\mathbf{f}(0)||^2$ depending on the initial condition $X_{h,0}\in\mathbf{X}_h$ and where the hidden constant does not depend on the final time $T_f$, the mesh size $h$, and the polynomial degrees $\ell, m$.
\end{theorem}
\begin{proof} 
We divide the proof into four steps. First, we use the inf-sup condition \eqref{eq:gen_inf_sup} to obtain a robust estimate on $\varphi_h$ which holds also when $d_0$ vanishes. Then, we derive the total energy balance associated to \eqref{eq:discrete_weak_form2}. In the third step, we estimate the right-hand side of the previous energy balance and in the fourth step we conclude.

\vspace{1mm}\textbf{Step 1}:
Taking $Y_h = (\mathbf{v_h}, 0, 0, 0)$ as test function in \eqref{eq:discrete_weak_form2} one has
$$
\mathcal{B}_h(\varphi_h, \mathbf{v}_h) = 
\mathcal{A}_h^e(\mathbf{u}_h, \mathbf{v}_h) - (\mathbf{f}, \mathbf{v}_h).
$$
Plugging the previous identity into \eqref{eq:gen_inf_sup}, using Lemma \eqref{lem:boundcoerc_bil_forms}, and applying the discrete Poincaré--Korn inequality \cite{Botti2020_korn, Brenner2004} it is inferred that
\begin{equation}
    \label{eq:bndphi_robust}
    \begin{aligned}
    \mathbb{B}^2||\varphi_h||^2 &\lesssim 
    \mathcal{D}_h(\varphi_h, \varphi_h) + \left(
    \underset{\mathbf{0} \neq \mathbf{v}_h \in \mathbf{V}^{\ell}_h}{\mbox{sup}} \frac{\mathcal{A}_h^e(\mathbf{u}_h, \mathbf{v}_h) - (\mathbf{f}, \mathbf{v}_h)}{||\mathbf{v}_h||_{DG,e}} \right)^2 \\ &\lesssim
    \mathcal{D}_h(\varphi_h, \varphi_h) + ||\mathbf{u}_h||_{DG,e}^2 + \mu_m^{-1}||\mathbf{f}||^2.
    \end{aligned}
\end{equation}
Therefore, it follows from the estimates \eqref{eq:positivity_Mh} and \eqref{eq:bndphi_robust} that
\begin{equation}
    \label{eq:bnd_L2terms}
    ||X_h||_{\mathcal{E}}^2 \lesssim
    \mathcal{M}_h(X_h, X_h) + ||\mathbf{u}_h||_{DG,e}^2 + \frac{||\mathbf{f}||^2}{\mu_m}
\end{equation}

\textbf{Step 2}:
Let $t\in(0,T_f]$. We test \eqref{eq:discrete_weak_form2} with $Y_h = (\dot{\mathbf{u}}_h, p_h, T_h, \varphi_h)$ to get
\begin{equation}
\begin{aligned}
    \label{eq:stab_est_prob}
    \mathcal{M}_h(\dot{X}_h, X_h) + &\mathcal{A}_h^e(\mathbf{u}_h,\dot{\mathbf{u}}_h)
    + \mathcal{A}_h^T(T_h,T_h) + \mathcal{C}_h(p_h,T_h)
    + \mathcal{A}_h^p(p_h,p_h)  \\
    &- \cancel{\mathcal{B}_h(\varphi_h,\dot{\mathbf{u}}_h)} +  \cancel{\mathcal{B}_h(\varphi_h,\dot{\mathbf{u}}_h)}   
    = (H,T_h) + (g,p_h) + (\mathbf{f},\dot{\mathbf{u}}_h).
    \end{aligned}
\end{equation}
Owing to the symmetry of the bilinear forms $\mathcal{M}_h$ and $\mathcal{A}_h^e$ we observe that 
$$
\left(\mathcal{M}_h(\dot{X}_h, X_h) + \mathcal{A}_h^e(\mathbf{u}_h,\dot{\mathbf{u}}_h)\right) (t)=
\frac12 \frac{d}{dt} \left(\mathcal{M}_h(X_h, X_h)(t) +
\mathcal{A}_h^e(\mathbf{u}_h,\mathbf{u}_h)(t)\right).
$$
Using the previous identity and integrating \eqref{eq:stab_est_prob} in time between $0$ and $t$, we have 
$$
\begin{aligned}
  \frac12 &\left(\mathcal{M}_h (X_h, X_h) +
\mathcal{A}_h^e(\mathbf{u}_h,\mathbf{u}_h)\right)(t) +
\int_0^t\hspace{-1mm}\left(\mathcal{A}_h^T(T_h,T_h) + \mathcal{C}_h(p_h,T_h)
    + \mathcal{A}_h^p(p_h,p_h)\right) (s)\, {\rm d}s \\
    &\;\qquad = \int_0^t \left( (H,T_h) + (g,p_h) + (\mathbf{f},\dot{\mathbf{u}}_h)\right) (s)\, {\rm d}s +\frac12\left(\mathcal{M}_h(X_h, X_h) +
    \mathcal{A}_h^e(\mathbf{u}_h,\mathbf{u}_h)\right)(0).
\end{aligned}
$$
We can bound from below the left-hand side of the previous energy balance by using \eqref{eq:positivity_Ah} and \eqref{eq:bnd_L2terms} to obtain
\begin{equation}
\begin{aligned}
    \label{eq:stab_step2}
    &||X_h(t)||_{\mathcal{E}}^2 
    +\int_0^t \left( ||T_h(s)||_{DG,T}^2+||p_h(s)||_{DG,p}^2\, \right){\rm d}s
    \lesssim \mathcal{M}_h(X_{h,0}, X_{h,0}) 
    \\&\qquad
    + ||\mathbf{u}_{h,0}||^2_{DG,e}
    +\frac{||\mathbf{f}(t)||^2}{\mu_m}
    +\int_0^t \left((H,T_h)+(g,p_h)+ (\mathbf{f},\dot{\mathbf{u}}_h)\right) (s)\, {\rm d}s
    \\&\qquad
    \lesssim \mathcal{M}_h(X_{h,0}, X_{h,0}) + ||\mathbf{u}_{h,0}||^2_{DG,e}+\mathcal{R}_1  +
    \mathcal{R}_2+\mathcal{R}_3,
    \end{aligned}
\end{equation}
with $X_{h,0}=(\mathbf{u}_{h,0}, p_{h,0}, T_{h,0}, \varphi_{h,0})$ corresponding to the initial condition of the semi-discrete problem \eqref{eq:discrete_weak_form2} and  
$$
\mathcal{R}_1 = \mu_m^{-1}||\mathbf{f}(t)||^2, \
\mathcal{R}_2 =\int_0^t  \left( (H,T_h)(s) + (g,p_h)(s)\right) \, {\rm d}s,\;\ \text{and }\ \mathcal{R}_3 = \int_0^t  (\mathbf{f},\dot{\mathbf{u}}_h) (s) \, {\rm d}s.
$$

\vspace{1mm}\textbf{Step 3}:
We proceed by bounding the terms in the right-hand side of \eqref{eq:stab_step2} starting with $\mathcal{R}_1$. Recalling the regularity assumption $\mathbf{f}\in H^1((0,T_f];L^2(\Omega))$ and using the fact that \\ $F(t) = F(0)+\int_0^t \dot{F}(s)\, {\rm d}s$ for all $F\in H^1([0,t])$, we have
\begin{equation}
    \label{eq:rhs_wp_semidiscr_R1}
\mathcal{R}_1 \le \mu_m^{-1}||\mathbf{f}(t)||^2 \lesssim \mu_m^{-1}\left(\int_0^t ||\dot{\mathbf{f}}(s)||^2 \, {\rm d}s + ||\mathbf{f}(0)||^2\right).
\end{equation}
In order to bound the term $\mathcal{R}_2$ use the Cauchy--Schwarz, discrete Poincar\'e, and Young inequality inequality to infer that
\begin{equation}
    \label{eq:rhs_wp_semidiscr_CSineq}
    \begin{aligned}
    \mathcal{R}_2 &\leq  \int_0^t \Big( ||H||\, ||T_h|| 
    + ||g||\, ||p_h||\Big)(s)\, {\rm d}s \\ &\lesssim
    \int_0^t \Big( \theta_m^{-\frac12}||H||\,||T_h||_{DG,T} +
    k_m^{-\frac12}||g||\,||p_h||_{DG,p}\Big)(s)\, {\rm d}s
    \\ &\lesssim
    \int_0^t \left(||T_h(s)||_{DG,T}^2+||p_h(s)||_{DG,p}^2\right)\, {\rm d}s + \int_0^t \left(\theta_m^{-1}||H(s)||^2 +
    k_m^{-1}||g(s)||^2\right)\, {\rm d}s
    \end{aligned}
\end{equation}
Concerning the term $\mathcal{R}_3$, since $\mathbf{f}\in H^1((0,T_f];L^2(\Omega))$, we are allowed to integrate by parts with respect to time and obtain
\begin{equation}
    \label{eq:rhs_wp_semidiscr_timeint}
    \begin{aligned}
    \mathcal{R}_3&=  \int_0^t  - (\dot{\mathbf{f}},\mathbf{u}_h)(s)\, {\rm d}s 
     + (\mathbf{f},\mathbf{u}_h)(t) - (\mathbf{f},\mathbf{u}_h)(0)
     \\ &\lesssim 
     \int_0^t t^{\frac12}||\dot{\mathbf{f}}(s)||\, t^{-\frac12}||\mathbf{u}_h||\, {\rm d}s + ||\mathbf{f}(t)||\ ||\mathbf{u}_h(t)|| +
     ||\mathbf{f}(0)||\ ||\mathbf{u}_{h,0}||
     \\ &\lesssim \int_0^t  \bigg( t||\dot{\mathbf{f}}(s)||^2 
     + \frac{||\mathbf{u}_h(s)||_{DG,e}^2}{t}\, \bigg) {\rm d}s  
     + ||\mathbf{u}_h(t)||_{DG,e}^2  
     + ||\mathbf{f}(0)||^2 + ||\mathbf{u}_{h,0}||_{DG,e}^2,
    \end{aligned}
\end{equation}
where, to pass to the third line, we have used again the discrete Korn--Poincaré inequality followed by the Young inequality and the second inequality in \eqref{eq:rhs_wp_semidiscr_R1}.
Moreover, taking the supremum for $s\in (0,t]$ in the second integrand in \eqref{eq:rhs_wp_semidiscr_timeint} leads to 
\begin{equation}
    \label{eq:rhs_R3_semidiscr}
    \mathcal{R}_3 \lesssim 
     t\int_0^t ||\dot{\mathbf{f}}(s)||^2\, {\rm d}s + \sup_{s\in (0,t]}||\mathbf{u}_h(s)||_{DG,e}^2 + 
     ||\mathbf{f}(0)||^2 + ||\mathbf{u}_{h,0}||_{DG,e}^2.
\end{equation}

\textbf{Step 4}:
To conclude, we plug the estimates \eqref{eq:rhs_wp_semidiscr_R1}, \eqref{eq:rhs_wp_semidiscr_CSineq}, and \eqref{eq:rhs_R3_semidiscr} into \eqref{eq:stab_step2} and we take the supremum for $t\in(0,\overline{t}]$, with $0<\overline{t}\le T_f$, to get
$$
    \begin{aligned}
    &\sup_{t\in (0,\overline{t}]} \left(||X_h||_{\mathcal{E}}^2\right) + \int_0^{\overline{t}} \left( ||T_h(s)||_{DG,T}^2+||p_h(s)||_{DG,p}^2 \right) \, {\rm d}s 
    \lesssim \frac12\sup_{t\in (0,\overline{t}]}||\mathbf{u}_h(s)||_{DG,e}^2 \\
    &\quad +\frac12 \int_0^{\overline{t}} \left( ||T_h(s)||_{DG,T}^2+||p_h(s)||_{DG,p}^2 \right)\, {\rm d}s + \int_0^{\overline{t}} \left( \frac{||H(s)||^2}{\theta_m} +
    \frac{||g(s)||^2}{k_m} \right) \, {\rm d}s \\ &\quad+ 
    \frac{\overline{t}\mu_m+1}{\mu_m}\int_0^{\overline{t}} ||\dot{\mathbf{f}}(s)||^2\, {\rm d}s 
    + \frac{1+\mu_m}{\mu_m}||\mathbf{f}(0)||^2 
    + \mathcal{M}_h(X_{h,0}, X_{h,0})  
    + ||\mathbf{u}_{h,0}||_{DG,e}^2
    \end{aligned}
$$
Rearranging the previous bound, it is inferred that
\begin{equation}
    \label{eq:stab_estimate_1}
    \begin{aligned}
    \sup_{t\in (0,\overline{t}]}& \left(||X_h||_{\mathcal{E}}^2\right) + \int_0^{\overline{t}} \left( ||T_h(s)||_{DG,T}^2+||p_h(s)||_{DG,p}^2 \right) \, {\rm d}s
    \\& \lesssim \mathcal{R}_0
    +\int_0^{\overline{t}} \left( \frac{||H(s)||^2}{\theta_m} +
    \frac{||g(s)||^2}{k_m} + \frac{\overline{t}\mu_m+1}{\mu_m}  ||\dot{\mathbf{f}}(s)||^2 \right)\, {\rm d}s, 
    \end{aligned}
\end{equation}
where we have defined $\mathcal{R}_0 = \mathcal{M}_h(X_{h,0}, X_{h,0}) +||\mathbf{u}_{h,0}||_{DG,e}^2 +(1+\mu_m^{-1})||\mathbf{f}(0)||^2$ only depending on the initial problem data. Since \eqref{eq:stab_estimate_1} holds for an  arbitrary $\overline{t}\in (0, T_f]$, this concludes the proof.
\end{proof}

\section{Error analysis}
\label{sec:error_est}
In this section we establish an a priori error estimate for the solution of the PolyDG semi-discrete problem \eqref{eq:discrete_weak_form2}. For the sake of simplicity, we decide not to explicitly track the dependencies of the inequality constants with respect to the model coefficients. Hence, in what follows, the constant hidden in the notation $x\lesssim y$ might depend on the thermo-poroelastic parameters and on $||\boldsymbol{\eta}||_{L^\infty(\Omega)}^d$.

We start by defining the $DG$-norms that will be needed in the error analysis
\begin{equation}
    \label{eq:DG_triple_norms}
    \begin{aligned}
    &|||S|||^2_{DG,T} = ||S||^2_{DG,T} + ||\sigma^{-\frac12} \avg{ \boldsymbol{\Theta}\nabla_h S}||_{\mathcal{F}}^2 \ &&  \forall \ S \in H^2({\mathcal{T}_h}),\\ 
    & |||q|||^2_{DG,p} = ||q||^2_{DG,p} + || \xi^{-\frac12} \avg{\mathbf{K} \nabla_h q}||_{\mathcal{F}}^2 \ && \forall \ q \in H^2({\mathcal{T}_h}),\\ 
    & |||\mathbf{v}|||^2_{DG,e} = ||\mathbf{v}||^2_{DG,e} +  ||\zeta^{-\frac12}\avg{\boldsymbol{\epsilon}_h(\mathbf{v})}||_{\mathcal{F}}^2 \ && \forall \ \mathbf{v} \in \mathbf{H}^2(\mathcal{T}_h), \\
    &|||\psi|||^2_{DG,\varphi} = ||\psi||^2 + ||\varrho^{\frac12} \avg{ \psi}||_{\mathcal{F}}^2 \ && \forall \ \psi \in H^1(\mathcal{T}_h).
    \end{aligned}
\end{equation}
Then, we introduce the interpolants  $X_I =\left(\mathbf{u}_I, p_I, T_I, \varphi_I \right) \in \mathbf{X}_h$ of the solution to the continuous formulation \eqref{eq:semi_discrete_cont}. 
In order to properly bound the interpolation errors, we define the Stein extension operator and state a result instrumental for the error analysis. For a polytopic mesh $\mathcal{T}_h$ satisfying Assumption~\ref{ass:A2}, the Stein operator $\mathcal{E}: H^n(\kappa) \rightarrow H^n(\mathbb{R}^d)$ is defined for any $\kappa \in \mathcal{T}_h$ and $m \in \mathbb{N}_0$ such that
\begin{equation}
    \label{eq:stein_operator}
    \mathcal{E}v \rvert_{\kappa} = v, \quad ||\mathcal{E}v||_{H^m(\mathbb{R}^d)} \lesssim ||v||_{H^m(\kappa)} \qquad \forall v \in H^m(\kappa).  
\end{equation}
The corresponding vector-valued version acts component-wise and is denoted in the same way. 
In what follows, for any $\kappa\in\mathcal{T}_h$, we will denote by $\mathcal{K}_\kappa$ the simplex belonging to $\mathcal{T}_h^*$ such that $\kappa\subset\mathcal{K}_\kappa$.
Then, the following approximation properties hold (see \cite[Lemma 3.6]{AntoniettiMazzieri2018}, \cite[Theorem 36]{Cangiani2017}, and \cite[Corollary 5.1]{antonietti2020_stokesDG} for the detailed proof): 
\begin{lemma}
    \label{lemma:interp}
    Let Assumption~\ref{ass:mesh_Th1} be fulfilled. For any 
    $S \in H^n(\mathcal{T}_h)$, $q \in H^n(\mathcal{T}_h)$, $\mathbf{w} \in \mathbf{H}^n(\mathcal{T}_h)$ with $n \geq 2$, there exist $S_I \in V_h^{\ell}$, $q_I \in V_h^{\ell}$, $\mathbf{w}_I \in \mathbf{V}_h^{\ell}$ such that
    \begin{equation}
        \begin{aligned}
        |||S - S_I|||_{DG,T}^2 \hspace{-0.05cm} + \hspace{-0.05cm} |||q - q_I|||_{DG,p}^2 & \hspace{-0.05cm} \lesssim \hspace{-0.05cm} \sum_{\kappa \in \mathcal{T}_h} \frac{h^{2 \min\{\ell_{\kappa} + 1, n \} - 2}}{\ell_{\kappa}^{2n-3}} \hspace{-0.05cm} \left( ||\mathcal{E}S||_{H^n(\mathcal{K}_\kappa)}^2 \hspace{-0.05cm} + \hspace{-0.05cm} ||\mathcal{E}q||_{H^n(\mathcal{K}_\kappa)}^2 \right)\\
        |||\mathbf{w} - \mathbf{w}_I|||_{DG,e}^2 &\lesssim \sum_{\kappa \in \mathcal{T}_h} \frac{h^{2 \min\{\ell_{\kappa} + 1, n \} - 2}}{\ell_{\kappa}^{2n-3}} \ ||\mathcal{E}\mathbf{w}||_{\mathbf{H}^n(\mathcal{K}_\kappa)}^2.
        \end{aligned}
    \end{equation}
    Moreover, let $\psi \in H^1(\mathcal{T}_h)$ be such that $(\mathcal{E}\psi)|_{\kappa} \in H^r(\kappa)$ for some $r \geq 1$ and for all $\kappa \in \mathcal{T}_h$. Then, there is $\psi_I\in Q_h^m$ such that
    \begin{equation}
        |||\psi - \psi_I|||^2_{DG,\varphi} + \mathcal{D}_h(\psi - \psi_I,\psi - \psi_I)
        \lesssim \sum_{\kappa \in \mathcal{T}_h} \frac{h^{2 \min\{m_{\kappa} + 1, r\}}}{m_{\kappa}^{2r}} \ ||\mathcal{E}\psi||_{H^r(\mathcal{K}_\kappa)}^2.
    \end{equation}
\end{lemma}

\subsection{Error equations}
\label{sec:error_equations}
We continue the error analysis with the derivation of the equations satisfied by the discretization errors $E = (\mathbf{e}^u, e^p, e^T, e^{\varphi})$, where
\begin{equation}
    \begin{aligned}
    & \textbf{e}^u(t) = \mathbf{u}(t) - \mathbf{u}_h(t), \quad e^p(t) = p(t) - p_h(t), \\
    & e^T(t) = T(t) - T_h(t), \quad e^{\varphi}(t) = \varphi(t) - \varphi_h(t),\\
    \end{aligned}
\end{equation}
with $X_h = (\mathbf{u}_h, p_h, T_h, \varphi_h)(t)$ and $X=(\mathbf{u}, p, T, \varphi)(t)$ denoting for all $t\in (0,T_f]$ the solutions to \eqref{eq:discrete_weak_form2} and \eqref{eq:semi_discrete_cont}, respectively.
We remark that the errors can be splitted as $E(t) = E_I(t)-E_h(t)$, where $E_I(t) = X(t)-X_I(t)$ and $E_h(t) = X_I(t)-X_h(t)$. 

In order to extend the bilinear forms defined in \eqref{eq:bilinear_forms_discr} to the space of continuous solutions we need further regularity requirements on $X=(\mathbf{u}, p, T, \varphi)$. In particular, we assume element-wise $H^2$-regularity of the displacement, temperature, and pressure together with the continuity of the normal stress, fluid flow, and heat flux across the interfaces $F\in\mathcal{F}_I$ for all time $t\in (0,T_f]$. We also require $\varphi\in H^1((0,T_f];H^1(\Omega))$, so that the stabilization term $\mathcal{D}_h$ in \eqref{eq:discrete_weak_form2} vanishes when tested with the exact solution. This further hypothesis on the regularity of the pseudo-total pressure can be inferred from the continuous formulation by reasoning as in \cite[Section 2.3.3]{Botti2021}.

Under the previous regularity assumptions, we can insert the exact solutions $(\mathbf{u}, p, T, \varphi)$ into \eqref{eq:discrete_weak_form2} obtaining a formulation equivalent to \eqref{eq:semi_discrete_cont}. Subtracting the resulting equation from problem \eqref{eq:discrete_weak_form2} defining the discrete solutions, we infer
\begin{equation}
    \label{eq:error_eq_text}
    \begin{aligned}
    & \mathcal{M}_h(\dot{E}, Y_h) + \mathcal{A}_h(E, Y_h) - \mathcal{B}_h(e^{\varphi},\mathbf{v}_h) +  \mathcal{B}_h(\psi_h,\dot{\mathbf{e}}^u)
    = 0
    \end{aligned}
\end{equation}
for all $ Y_h = (\mathbf{v}_h, q_h, S_h, \psi_h) \in \mathbf{X}_h$. Here we assume that problem \eqref{eq:discrete_weak_form2} is supplemented by initial conditions $X_{h,0} = (\mathbf{u}_I(0), p_I(0), T_I(0), \varphi_I(0))$, where $\mathbf{u}_I, p_I, T_I, \varphi_I$ are the interpolants given by Lemma~\ref{lemma:interp}, so that the error equation \eqref{eq:error_eq_text} is completed by the condition $E_h(0) = \mathbf{0}$. We now test \eqref{eq:error_eq_text} against $\left(\dot{\mathbf{e}}_h^u, e_h^p, e_h^T, e_h^{\varphi} \right)$ and use the linearity of the bilinear forms to obtain
\begin{equation}
    \label{eq:error_eq3_text}
    \begin{aligned}
    & \mathcal{M}(\dot{E}_h, E_h) + \mathcal{A}_h^{T}(e_h^T,e_h^T) + \mathcal{C}_h(e_h^p,e_h^T) + \mathcal{A}_h^{p}(e_h^p,e_h^p) + \mathcal{A}_h^{e}(\mathbf{e}_h^u,\dot{\mathbf{e}}_h^u) = \mathcal{M}(\dot{E}_I, E_h) \\
    & + \mathcal{A}_h^{T}(e_I^T,e_h^T) + \mathcal{C}_h(e_I^p,e_h^T) + \mathcal{A}_h^{p}(e_I^p,e_h^p) + \mathcal{A}_h^{e}(\mathbf{e}_I^u, \dot{\mathbf{e}}_h^u) - \mathcal{B}_h(e_I^{\varphi},\dot{\mathbf{e}}_h^u) +  \mathcal{B}_h(e_h^{\varphi},\dot{\mathbf{e}}_I^u) 
    \end{aligned}
\end{equation}
The previous identity is the starting point for the error estimate of the next section.

\subsection{Error estimate}
\label{sec:error_estimate}
In this section we derive the a priori estimate for the semi-discrete problem \eqref{eq:discrete_weak_form2}. Before doing so, we provide an instrumental result establishing the boundedness of the discrete bilinear forms defined in \eqref{eq:bilinear_forms_discr}.
\begin{lemma}
\label{lemma:bound_bilinearform_error}
Let Assumption~\ref{assumption:model_problem} and Assumption~\ref{ass:mesh_Th1} be satisfied and assume that the polynomial degrees of the PolyDG approximation satisfy $m\leq\ell + 1$. Then, 
\begin{equation}
    \begin{aligned}
    \mathcal{A}_h^T(T,S) \lesssim \ & |||T|||_{DG,T} ||S||_{DG,T} \quad && \forall \ T \in H^2(\mathcal{T}_h), \ \forall \ S \in V_h^{\ell},\\
    \mathcal{A}_h^p(p,q) \lesssim \ & |||p|||_{DG,p} ||q||_{DG,p} \quad && \forall \ p \in H^2(\mathcal{T}_h), \ \forall \ q \in V_h^{\ell},\\
    \mathcal{A}_h^e(\mathbf{u},\mathbf{v}) \lesssim \ & |||\mathbf{u}|||_{DG,e} ||\mathbf{v}||_{DG,e} \quad && \forall \ \mathbf{u} \in \mathbf{H}^2(\mathcal{T}_h),\ \forall \ \mathbf{v} \in \mathbf{V}_h^{\ell},\\
    \mathcal{C}_h(p,S) \lesssim \ & |||p|||_{DG,p} ||S||_{DG,T} \quad && \forall \ p \in H^2(\mathcal{T}_h), \ \forall \ S \in V_h^{\ell},\\
    \mathcal{B}_h(\psi,\mathbf{u}) \lesssim \ & ||| \psi |||_{DG,\psi} || \mathbf{u} ||_{DG,e} \quad && \forall \ \psi \in H^1(\mathcal{T}_h), \ \forall \ \mathbf{u} \in \mathbf{V}^{\ell}_h,\\
    \mathcal{B}_h(\psi,\mathbf{u}) \lesssim \ & || \psi || \ ||| \mathbf{u} |||_{DG,e} \quad && \forall \ \psi \in Q_h^m, \ \forall \ \mathbf{u} \in \mathbf{H}^2(\mathcal{T}_h),\\
    \end{aligned}
\end{equation}
\end{lemma}
\begin{proof}
For the boundedness of the bilinear form $\mathcal{A}^T_h, \mathcal{A}^p_h, \mathcal{A}^e_h$ we refer to \cite[Lemma 3.1]{Antonietti.Bonaldi:20} and \cite[Lemma A.2]{Antonietti.Botti.ea:21}.
The boundedness of $\mathcal{C}_h$ directly follows from the application of Cauchy-Schwarz inequality and the definitions of the norms $|||\cdot|||_{DG,p}, ||\cdot||_{DG,T}$. Concerning the first boundedness property of $\mathcal{B}_h$, we let $\psi \in H^1(\mathcal{T}_h), \ \mathbf{u} \in \mathbf{V}^{\ell}_h$ and apply twice the Cauchy--Schwarz inequality to obtain
\begin{equation}
\label{eq:bnd_Bh1}
    \begin{aligned}
    \mathcal{B}_h(\psi,\mathbf{u}) & \leq |(\psi, \nabla_h\cdot\mathbf{u})| + \left| \sum_{F \in \mathcal{F}} \int_F \avg{\psi}\jump{\mathbf{v}}_n \right|\\ 
    &\leq ||\psi||\ ||\nabla_h\mathbf{u}|| 
    + \left(\sum_{F \in \mathcal{F}} \varrho|| \avg{\psi}||_F^2\right)^{\frac12}
    \left(\sum_{F \in\mathcal{F}}\varrho^{-1}
    ||\jump{\mathbf{v}}_n||_F^2\right)^{\frac12}.
    \end{aligned}
\end{equation}
Additionally, we remark that
\begin{equation}
    \varrho^{-1} = \alpha_4^{-1} \max_{\kappa \in \{\kappa^+, \kappa^-\}} \left(\frac{m}{h_{\kappa}} \right) \lesssim \max_{\kappa \in \{\kappa^+, \kappa^-\}} \left(\frac{m}{\ell^2} \right) \zeta \lesssim \zeta,
\end{equation}
where in the last inequality we have used the hypothesis $m\le \ell + 1$. Therefore, as a result of \eqref{eq:bnd_Bh1} and the definition of the $DG$-norms we have
\begin{equation}
    \mathcal{B}_h(\psi, \mathbf{u}) \lesssim |||\psi|||_{DG,\psi} \|\mathbf{u}\|_{DG,e}.
\end{equation}
Finally, we consider the boundedness of $\mathcal{B}_h$ in the case $\psi \in Q^m_h,\ \mathbf{u} \in \mathbf{H}^2(\mathcal{T}_h)$. Proceeding as in \eqref{eq:bnd_Bh1} we obtain
\begin{equation}
    \begin{aligned}
    \mathcal{B}_h(\psi,\mathbf{u}) & \leq ||\psi||\, ||\nabla_h\mathbf{u}||\hspace{-0.5mm} 
    + \hspace{-0.5mm}\left(\sum_{F \in \mathcal{F}} \frac{|| \avg{\psi}||_F^2}{\zeta}\right)^{\frac12} \hspace{-1mm}
    \left(\sum_{F \in\mathcal{F}}\zeta
    ||\jump{\mathbf{v}}_n||_F^2\right)^{\frac12} \hspace{-1mm} \lesssim \|\psi\|\, |||\mathbf{u}|||_{DG,e},
    \end{aligned}
\end{equation}
where the final bound results from the discrete trace inequality \eqref{eq:trace_inverse_ineq}.
\end{proof}

We are now ready to state the main result of this section:
\begin{theorem}
\label{thm:error_est_poinc}
Let the assumptions of Theorem~\ref{thm:stability_est} hold. Let the solution $X=(\mathbf{u}, p, T, \varphi) \\ \in C^0((0,T]; \mathbf{V}\times V\times V\times Q)$ of problem \eqref{eq:semi_discrete_cont} satisfy the additional regularity
\begin{equation}
    \begin{aligned}
    X \in \ &  C^1\left((0,T];\ \mathbf{H}^j(\mathcal{T}_h)\times H^k(\mathcal{T}_h)\times H^n(\mathcal{T}_h)\times (H^p(\mathcal{T}_h)\cap H^1(\Omega))\right), 
    \end{aligned}
\end{equation}
with $j,k,n,p \geq 2$ and let $X_h=(\mathbf{u}_h, p_h, T_h, \varphi_h) \in C^1((0,T]; \mathbf{X}_h)$ be the solution of problem \eqref{eq:discrete_weak_form2}. Then, for any $t \in (0,T_f]$, the error $E_h(t) = (\mathbf{e}_h^u, e_h^p, e_h^T, e_h^{\varphi})$ satisfies
\begin{equation}
    \label{eq:error_Eh_final2}
    \begin{aligned}
    \|E_h(t)\|_{\mathcal{E}}^2 &+ \hspace{-0.15cm} \int_0^t \hspace{-0.15cm} \bigg( ||e_h^T(s)||_{DG,T}^2 + ||e_h^p(s)||_{DG,p}^2 \bigg){\rm d}s \\ &\lesssim \sum_{\kappa \in \mathcal{T}_h} \frac{h^{2 \min\{\ell_{\kappa} + 1, j \} - 2}}{\ell_{\kappa}^{2j-3}} \ \bigg[ ||\mathcal{E}\mathbf{u}||_{\mathbf{H}^j(\mathcal{K})}^2 + \int_0^t ||\mathcal{E}\dot{\mathbf{u}}(s)||_{\mathbf{H}^j(\mathcal{K})}^2ds \bigg]\\
    & + \sum_{\kappa \in \mathcal{T}_h} \frac{h^{2 \min\{\ell_{\kappa} + 1, k \} - 2}}{\ell_{\kappa}^{2k-3}} \ \int_0^t \bigg[ ||\mathcal{E}p(s)||_{H^k(\mathcal{K})}^2 + ||\mathcal{E}\dot{p}(s)||_{H^k(\mathcal{K})}^2 \bigg]ds \\
    & + \sum_{\kappa \in \mathcal{T}_h} \frac{h^{2 \min\{\ell_{\kappa} + 1, n \} - 2}}{\ell_{\kappa}^{2n-3}} \ \int_0^t \bigg[ ||\mathcal{E}T(s)||_{H^n(\mathcal{K})}^2 + ||\mathcal{E}\dot{T}(s)||_{H^n(\mathcal{K})}^2 \bigg]ds \\
    & + \sum_{\kappa \in \mathcal{T}_h} \frac{h^{2 \min\{m_{\kappa} + 1, p \}}}{m_{\kappa}^{2p}} \ \bigg[ ||\mathcal{E}\varphi||_{H^p(\mathcal{K})}^2 + \int_0^t ||\mathcal{E}\dot{\varphi}(s)||_{H^p(\mathcal{K})}^2ds \bigg],\\
    \end{aligned}
\end{equation}
where the hidden constant depends on the time $t$ and on the material properties, but are independent of the discretization parameters.
\end{theorem}
\begin{proof}
Owing to the symmetry of $\mathcal{M}_h$ and $A^e_h$ and moving the time derivative from the discretization error to the interpolation one, we rewrite equation \eqref{eq:error_eq3_text} as
\begin{equation}
    \label{eq:error_eq4_text}
    \begin{aligned}
    & \frac{1}{2}\frac{d}{dt} \bigg( \mathcal{M}(E_h, E_h) + \mathcal{A}_h^{e}(\mathbf{e}_h^u,\dot{\mathbf{e}}_h^u) \bigg) +  \mathcal{A}_h^{T}(e_h^T,e_h^T) + \mathcal{C}_h(e_h^p,e_h^T) + \mathcal{A}_h^{p}(e_h^p,e_h^p)  \\
    & = \mathcal{M}(\dot{E}_I, E_h) + \mathcal{A}_h^{T}(e_I^T,e_h^T) + \mathcal{C}_h(e_I^p,e_h^T) + \mathcal{A}_h^{p}(e_I^p,e_h^p) + \frac{d}{dt}\mathcal{A}_h^{e}(\mathbf{e}_I^u, \mathbf{e}_h^u) \\
    & - \mathcal{A}_h^{e}(\dot{\mathbf{e}}_I^u,\mathbf{e}_h^u) - \frac{d}{dt} \mathcal{B}_h(e_I^{\varphi},\mathbf{e}_h^u) + \mathcal{B}_h(\dot{e}_I^{\varphi},\mathbf{e}_h^u) +  \mathcal{B}_h(e_h^{\varphi},\dot{\mathbf{e}}_I^u) 
    \end{aligned}
\end{equation}
Then, integrating with respect to time between $0$ and $t \leq T_f$, recalling that $E_h(0) = \mathbf{0}$, and proceeding as in the proof of Theorem~\ref{thm:stability_est}, we get
\begin{equation}
    \label{eq:error_eq5_text}
    \begin{aligned}
    \|E_h(t)\|_{\mathcal{E}}^2 + \mathcal{D}_h(e^{\varphi}_h(t), e^{\varphi}_h(t)) & + \int_0^t \left( ||e_h^T(s)||_{DG,T}^2 + ||e_h^p(s)||_{DG,p}^2 \right) {\rm d}s \\
    & \lesssim \mathfrak{R}_1(t) + \int_0^t \big( \mathfrak{R}_2(s) +  \mathfrak{R}_3(s) + \mathfrak{R}_4(s) \big) {\rm d}s,
    \end{aligned}
\end{equation}    
where the terms in the right-hand side are given by
\begin{equation}
    \begin{aligned}
    & \mathfrak{R}_1 = \mathcal{A}^e_h(\mathbf{e}^u_I, \mathbf{e}^u_h) - \mathcal{B}_h(e^{\varphi}_I, \mathbf{e}^u_h),\quad 
    && \mathfrak{R}_3 = \mathcal{A}_h^{T}(e_I^T,e_h^T)
    + \mathcal{C}_h(e_I^p,e_h^T) + \mathcal{A}_h^{p}(e_I^p,e_h^p), \\
    & \mathfrak{R}_2 = \mathcal{M}_h(\dot{E}_I, E_h), \quad 
    && \mathfrak{R}_4 = - \mathcal{A}^e_h(\dot{\mathbf{e}}^u_I,\mathbf{e}^u_h) + \mathcal{B}_h(\dot{e}_I^{\varphi},\mathbf{e}_h^u) +  \mathcal{B}_h(e_h^{\varphi},\dot{\mathbf{e}}_I^u).\\
    \end{aligned}
\end{equation}
We bound $\mathfrak{R}_1, \mathfrak{R}_2, \mathfrak{R}_3, \mathfrak{R}_4$ by the repeated use of standard inequalities (i.e. Cauchy-Schwarz, Young, and triangle inequalities) and Lemma~\ref{lemma:bound_bilinearform_error} to obtain
\begin{equation}
    \begin{aligned}
    \mathfrak{R}_1 \lesssim \ & \|\mathbf{e}^u_h\|^2_{DG,e} + \big( |||\mathbf{e}^u_I|||^2_{DG,e} + |||e^{\varphi}_I|||^2_{DG,\varphi} \big), \\
    \mathfrak{R}_2 \lesssim \ & \left(\|E_h\|^2_{\mathcal{E}} + \mathcal{D}_h(e^{\varphi}_h, e^{\varphi}_h) \right) + \left( \|\dot{e}^p_I\|^2 + \|\dot{e}^T_I\|^2 + \|\dot{e}^{\varphi}_I\|^2 + \mathcal{D}_h(\dot{e}^{\varphi}_I, \dot{e}^{\varphi}_I) \right) ,\\
    \mathfrak{R}_3 \lesssim \ & \big( \|e^p_h\|^2_{DG,p} + \|e^T_h\|^2_{DG,T} \big) + \big( |||e^p_I|||^2_{DG,p} + |||e^T_I|||^2_{DG,T} \big),\\
    \mathfrak{R}_4 \lesssim \ & ( \|\mathbf{e}^u_h\|^2_{DG,e} + \|e^{\varphi}_h\|^2 ) + \big( |||\dot{\mathbf{e}}^u_I|||^2_{DG,e} + |||\dot{e}^{\varphi}_I|||^2_{DG,\varphi} \big) .
    \end{aligned}
\end{equation}
Exploiting the previous bounds and applying Gr\"onwall's Lemma \cite{Quarteroni2014} we infer that
\begin{equation}
    \label{eq:error_eq6_text}
    \begin{aligned}
    \|E_h(t)\|_{\mathcal{E}}^2 &+ \mathcal{D}_h(e^{\varphi}_h(t), e^{\varphi}_h(t)) + \int_0^t \bigg( ||e_h^T(s)||_{DG,T}^2 + ||e_h^p(s)||_{DG,p}^2 \bigg){\rm d}s \lesssim |||\mathbf{e}^u_I(t)|||^2_{DG,e} \\
    & +  |||e^{\varphi}_I(t)|||^2_{DG,\varphi} + \int_0^t \bigg( |||\dot{\mathbf{e}}^u_I(s)|||^2_{DG,e} + \|\dot{e}^p_I(s)\|^2 + \|\dot{e}^T_I(s)\|^2
    + \|\dot{e}^{\varphi}_I(s)\|^2 \\
    & + |||e^p_I(s)|||^2_{DG,p} + |||e^T_I(s)|||^2_{DG,T} + |||\dot{e}^{\varphi}_I(s)|||^2_{DG,\varphi}
    + \mathcal{D}_h(\dot{e}^{\varphi}_I(s), \dot{e}^{\varphi}_I(s)) \bigg){\rm d}s
    \end{aligned}
\end{equation}   
Finally, the thesis follows by using Lemma~\ref{lemma:interp} to bound the interpolation errors.
\end{proof}

\section{Numerical results}
\label{sec:numerical_test}
We now assess the performance of the method in terms of accuracy and robustness. In order to observe the behaviour of the proposed scheme we design suitable problems starting from analytical manufactured solution. We show the results for both a linear steady case (i.e., $c_f = 0$ in \eqref{eq:QS_TPE_system}) and for the original non-linear TPE problem. In both the linear and non-linear cases the PolyDG spatial discretization is coupled with the (implicit) $\theta$-method for the integration in time. We consider $\theta \geq 1/2$ and recall that the case $\theta = 1/2$ yields a  second-order accurate scheme. The steady problem is considered as one step of the time-integration method with $\Delta t = 1$. For the non-linear problem we make use of the fixed-point iterative algorithm described in Section~\ref{sec:linearization_wp} in order to treat the convective transport term. 

Finally, the last numerical experiment deals with a fluid injection-extraction problem, inspired by a geothermal energy production configuration. For all the numerical test we used polygonal Voronoi meshes generated with the \texttt{Polymesher} algorithm \cite{Talischi2012}. In all the presented tests we consider $\ell = m$ as polynomial degree of the approximation for all the four variables of our formulation. Thus, for the sake of simplicity, we make use only of the symbol $\ell$ to denote the polynomial degree.

\subsection{Convergence analysis}
\label{sec:convergence_analysis}
We consider a square domain $\Omega = (0,2)^2$ and manufactured exact solution:
\begin{equation}
      \begin{aligned}
      u_1(\textbf{x},t) &= (e^{t} - 1) \left( \sin(2\pi y) (\cos(2 \pi x) - 1) + \frac1{\mu +\lambda} \sin(\pi x) \sin(\pi y) \right), \\
      u_2(\textbf{x},t) &= (e^{t} - 1) \left( \sin(2\pi x) (1 - \cos(2 \pi y)) + \frac1{\mu +\lambda} \sin(\pi x) \sin(\pi y) \right), \\
      p(\textbf{x},t) &= (e^{t} - 1) \sin(\pi x) \sin(\pi y), \\
      T(\textbf{x},t) &= (e^{t} - 1) (\cos(2 \pi x) - 1) (\cos(2 \pi y) - 1).
    \end{aligned}
\end{equation}
with initial conditions, boundary conditions, and forcing terms that are inferred from the exact solutions. The model coefficients are chosen as reported in Table~\ref{tab:TPE_params_convtest}.
For the linear case we have considered a sequence of polygonal meshes whose number of elements is $N = \{100, 310, 1000, 3100, 10000\}$, and $\ell = \{1, 3\}$, while, for what concerns the non-linear case, we have considered the following discretization parameters: $N = \{20, 80, 320, 1280, 5120 \}, \ell = 2, T_f = 0.1,$ and $\Delta t = $\num[exponent-product=\ensuremath{\cdot}]{5e-5}.
\begin{table}[H]
    \centering 
    \footnotesize
    \begin{tabular}{l | c c l | c c l | c}
    $a_0 \ [\si[per-mode = symbol]{\giga\pascal \per \kelvin\squared}]$ & 0.02 & & $\alpha \ [-]$ & 1 & & $\lambda \ [\si{\giga\pascal}]$ & 5  \\
    $b_0 \ [\si{\per \kelvin}]$ & 0.01 & & $\beta \ [\si{\giga\pascal \per \kelvin}]$ & 0.8 & & $\mathbf{K} \ [\si{\dm \squared \per \giga\pascal \per \hour}]$ & 0.2 \\
    $c_0 \ [\si{\per \giga\pascal}]$ & 0.03 & & $\mu \ [\si{\giga\pascal}]$ & 1 & & $\boldsymbol{\Theta} \ [\si{\dm \squared \giga\pascal \per \kelvin\squared \per \hour}]$ & 0.05
    \end{tabular}
\caption{Problem's parameters for the convergence analysis}
\label{tab:TPE_params_convtest}
\end{table}

In the following tables, $h$ denotes the mesh size and $roc$ stands for the computed rate of convergence. In both the linear and non-linear test cases, we can observe by looking at Figure~\ref{fig:linearsteadyconv_p1}, Figure~\ref{fig:linearsteadyconv_p3}, and Figure~\ref{fig:nonlinearunsteadyconv} that our results are in accordance with the expected error rates, cf. Theorem~\ref{thm:error_est_poinc}. Indeed, in the energy norm we observe a computed convergence rate of $1$ and $3$ using $\ell=1$ and $\ell=3$ approximations. Moreover, we observe a convergence rate of order $2$ and $4$ in $L^2$ norm using $\ell=1$ and $\ell=3$ approximations, respectively. For what concerns the non-linear test we observe that, as for the linear problem, we satisfy the estimates presented in Section~\ref{sec:error_est}. 
For the displacement and the temperature we also achieve $\ell + 1$ accuracy in $L^2$-norm, while for the pressure we see that the order of accuracy overcome the estimate provided by the theory, but in the last refinement we observe a slight loss with respect to the $\ell+1$ order. In the non-linear case, we observe that an average of $4$ fixed-point iterations at each time-step are required to reach the desired convergence, namely below a tolerance ensuring that the linearization error is smaller than the discretization one.
\begin{figure}[H]

\begin{subfigure}{1\textwidth}
\centering
\begin{subfigure}[b]{0.49\textwidth}
\begin{tikzpicture}
\begin{axis}[%
width=0.65\textwidth,
height=0.5\textwidth,
at={(0\textwidth,0\textwidth)},
scale only axis,
xmode=log,
xmin=2,
xmax=30,
xminorticks=true,
xlabel={$1/h$},
ymode=log,
ymin=1.1e-03,
ymax=1.5,
yminorticks=true,
ylabel={$L^2$-Errors},
legend style={draw=none,fill=none,legend cell align=left},
legend pos=north east
]
\addplot [color=red,solid,line width=1.0pt, mark=diamond*,mark options={color=red}]
  table[row sep=crcr]{
    2.6707   0.55085 \\
    4.639   0.19096 \\
    8.3601   0.061182 \\
    14.2688   0.019251 \\
    24.228   0.005867 \\
};
\addlegendentry{$\mathbf{u}(\mathbf{x},t)$}

\addplot [color=blue,solid,line width=1.0pt,mark=square*,mark options={color=blue}]
  table[row sep=crcr]{
    2.6707   0.2152 \\
    4.639   0.075634 \\
    8.3601   0.021812 \\
    14.2688   0.0065419 \\
    24.228   0.0019798 \\
};
\addlegendentry{$p(\mathbf{x},t)$}

\addplot [color=green,solid,line width=1.0pt,mark=triangle*,mark options={color=green}]
  table[row sep=crcr]{
    2.6707   0.21171 \\
    4.639   0.074632 \\
    8.3601   0.024798 \\
    14.2688   0.0081614 \\
    24.228   0.0025893 \\
};
\addlegendentry{$T(\mathbf{x},t)$}

\addplot [color=black,solid,line width=0.5pt]
  table[row sep=crcr]{
 14.2688     0.0289 \\
 24.228      0.0289 \\
 24.228      0.0085 \\
 14.2688     0.0289 \\ 
};

\node[right, align=left, text=black, font=\footnotesize]
at (axis cs:24.228,0.0187) {2}; 

\end{axis}
\end{tikzpicture}
\end{subfigure}
\begin{subfigure}[b]{0.49\textwidth}
\begin{tikzpicture}
\begin{axis}[%
width=0.65\textwidth,
height=0.5\textwidth,
at={(0\textwidth,0\textwidth)},
scale only axis,
xmode=log,
xmin=2,
xmax=30,
xminorticks=true,
xlabel={$1/h$},
ymode=log,
ymin=0.01,
ymax=10,
yminorticks=true,
ylabel={$DG$-Errors},
legend style={draw=none,fill=none,legend cell align=left},
legend pos=north east
]
\addplot [color=red,solid,line width=1.0pt, mark=diamond*,mark options={color=red}]
  table[row sep=crcr]{
    2.6707   2.9102 \\
    4.639   1.0209 \\
    8.3601   0.40921 \\
    14.2688   0.20336 \\
    24.228   0.10158 \\
};
\addlegendentry{$\mathbf{u}(\mathbf{x},t)$}

\addplot [color=blue,solid,line width=1.0pt,mark=square*,mark options={color=blue}]
  table[row sep=crcr]{
    2.6707   0.55821 \\
    4.639   0.20942 \\
    8.3601   0.075473 \\
    14.2688   0.036369 \\
    24.228   0.018586 \\
};
\addlegendentry{$p(\mathbf{x},t)$}

\addplot [color=green,solid,line width=1.0pt,mark=triangle*,mark options={color=green}]
  table[row sep=crcr]{
    2.6707   0.31126 \\
    4.639   0.12234 \\
    8.3601   0.054146 \\
    14.2688   0.027288 \\
    24.228   0.013442 \\
};
\addlegendentry{$T(\mathbf{x},t)$}

\addplot [color=black,solid,line width=0.5pt]
  table[row sep=crcr]{
 14.2688    0.2377 \\
 24.228     0.2377 \\
 24.228     0.14 \\
 14.2688    0.2377 \\ 
};
\node[right, align=left, text=black, font=\footnotesize]
at (axis cs:24.228,0.1889) {1}; 

\end{axis}
\end{tikzpicture}
\end{subfigure}
\caption{Convergence test - linear steady problem - polynomial degree $\ell = 1$.}
\label{fig:linearsteadyconv_p1}
\end{subfigure}
\\[10pt]


\begin{subfigure}{1\textwidth}
\centering
\begin{subfigure}[b]{0.49\textwidth}
\begin{tikzpicture}
\begin{axis}[%
width=0.65\textwidth,
height=0.5\textwidth,
at={(0\textwidth,0\textwidth)},
scale only axis,
xmode=log,
xmin=2,
xmax=30,
xminorticks=true,
xlabel={$1/h$},
ymode=log,
ymin=1e-08,
ymax=0.002,
yminorticks=true,
ylabel={$L^2$-Errors},
legend style={draw=none,fill=none,legend cell align=left},
legend pos=north east
]
\addplot [color=red,solid,line width=1.0pt, mark=diamond*,mark options={color=red}]
  table[row sep=crcr]{
    2.6707   0.00074211 \\
    4.639   6.762e-05 \\
    8.3601   6.1725e-06 \\
    14.2688   6.495e-07 \\
    24.228   6.0119e-08 \\
};
\addlegendentry{$\mathbf{u}(\mathbf{x},t)$}

\addplot [color=blue,solid,line width=1.0pt,mark=square*,mark options={color=blue}]
  table[row sep=crcr]{
    2.6707   0.00045648 \\
    4.639   4.4223e-05 \\
    8.3601   3.7283e-06 \\
    14.2688   3.82e-07 \\
    24.228   3.7281e-08 \\
};
\addlegendentry{$p(\mathbf{x},t)$}

\addplot [color=green,solid,line width=1.0pt,mark=triangle*,mark options={color=green}]
  table[row sep=crcr]{
    2.6707   0.00045615 \\
    4.639   4.2417e-05 \\
    8.3601   3.8414e-06 \\
    14.2688   4.0731e-07 \\
    24.228   3.922e-08 \\
};
\addlegendentry{$T(\mathbf{x},t)$}

\addplot [color=black,solid,line width=0.5pt]
  table[row sep=crcr]{
 14.2688    1.0188e-06 \\
 24.228     1.0188e-06 \\
 24.228     1.5e-07 \\
 14.2688    1.0188e-06 \\ 
};

\node[right, align=left, text=black, font=\footnotesize]
at (axis cs:24.228,5.8439e-07) {4}; 

\end{axis}
\end{tikzpicture}
\end{subfigure}
\begin{subfigure}[b]{0.49\textwidth}
\begin{tikzpicture}
\begin{axis}[%
width=0.65\textwidth,
height=0.5\textwidth,
at={(0\textwidth,0\textwidth)},
scale only axis,
xmode=log,
xmin=2,
xmax=30,
xminorticks=true,
xlabel={$1/h$},
ymode=log,
ymin=6e-6,
ymax=0.8,
yminorticks=true,
ylabel={$DG$-Errors},
legend style={draw=none,fill=none,legend cell align=left},
legend pos=north east
]
\addplot [color=red,solid,line width=1.0pt, mark=diamond*,mark options={color=red}]
  table[row sep=crcr]{
    2.6707   0.10736 \\
    4.639   0.01792 \\
    8.3601   0.0030304 \\
    14.2688   0.00054347 \\
    24.228   9.1862e-05 \\
};
\addlegendentry{$\mathbf{u}(\mathbf{x},t)$}

\addplot [color=blue,solid,line width=1.0pt,mark=square*,mark options={color=blue}]
  table[row sep=crcr]{
    2.6707   0.024036 \\
    4.639   0.0041376 \\
    8.3601   0.00066662 \\
    14.2688   0.00011438 \\
    24.228   1.9981e-05 \\
};
\addlegendentry{$p(\mathbf{x},t)$}

\addplot [color=green,solid,line width=1.0pt,mark=triangle*,mark options={color=green}]
  table[row sep=crcr]{
    2.6707   0.011937 \\
    4.639   0.0019137 \\
    8.3601   0.00033299 \\
    14.2688   6.0426e-05 \\
    24.228   1.0674e-05 \\
};
\addlegendentry{$T(\mathbf{x},t)$}

\addplot [color=black,solid,line width=0.5pt]
  table[row sep=crcr]{
 14.2688    0.0010 \\
 24.228    0.0010 \\
 24.228    2e-4 \\
 14.2688    0.0010 \\ 
};
\node[right, align=left, text=black, font=\footnotesize]
at (axis cs:24.228,6.0939e-04) {3}; 

\end{axis}
\end{tikzpicture}
\end{subfigure}
\caption{Convergence test - linear steady problem - polynomial degree $\ell = 3$.}
\label{fig:linearsteadyconv_p3}
\end{subfigure}
\\[10pt]

\begin{subfigure}{1\textwidth}
\centering
\begin{subfigure}[b]{0.49\textwidth}
\begin{tikzpicture}
\begin{axis}[%
width=0.65\textwidth,
height=0.5\textwidth,
at={(0\textwidth,0\textwidth)},
scale only axis,
xmode=log,
xmin=1,
xmax=23,
xminorticks=true,
xlabel={$1/h$},
ymode=log,
ymin=1e-07,
ymax=0.03,
yminorticks=true,
ylabel={$L^2$-Errors},
legend style={draw=none,fill=none,legend cell align=left},
legend pos=north east
]
\addplot [color=red,solid,line width=1.0pt, mark=diamond*,mark options={color=red}]
  table[row sep=crcr]{
  1.2627   0.0029659 \\
  2.4436   0.0002484 \\
  4.546   2.2459e-05 \\
  8.8342   2.2117e-06 \\
  17.8826   2.4816e-07 \\
};
\addlegendentry{$\mathbf{u}(\mathbf{x},t)$}

\addplot [color=blue,solid,line width=1.0pt,mark=square*,mark options={color=blue}]
  table[row sep=crcr]{
  1.2627   0.0037469 \\
  2.4436   0.00036747 \\
  4.546   4.0639e-05 \\
  8.8342   3.3337e-06 \\
  17.8826   6.0689e-07 \\
};
\addlegendentry{$p(\mathbf{x},t)$}

\addplot [color=green,solid,line width=1.0pt,mark=triangle*,mark options={color=green}]
  table[row sep=crcr]{
  1.2627   0.010367 \\
  2.4436   0.0010016 \\
  4.546   0.00010783 \\
  8.8342   9.5788e-06 \\
  17.8826   1.2064e-06 \\
};
\addlegendentry{$T(\mathbf{x},t)$}

\addplot [color=black,solid,line width=0.5pt]
  table[row sep=crcr]{
 8.8342     1.8218e-05 \\
 17.8826    1.8218e-05 \\
 17.8826    3e-06 \\
 8.8342     1.8218e-05 \\ 
};

\node[right, align=left, text=black, font=\footnotesize]
at (axis cs:17.8826,1.0609e-05) {3}; 

\end{axis}
\end{tikzpicture}
\end{subfigure}
\begin{subfigure}[b]{0.49\textwidth}
\begin{tikzpicture}
\begin{axis}[%
width=0.65\textwidth,
height=0.5\textwidth,
at={(0\textwidth,0\textwidth)},
scale only axis,
xmode=log,
xmin=1,
xmax=23,
xminorticks=true,
xlabel={$1/h$},
ymode=log,
ymin=1.5e-05,
ymax=0.09,
yminorticks=true,
ylabel={$DG$-Errors},
legend style={draw=none,fill=none,legend cell align=left},
legend pos=north east
]
\addplot [color=red,solid,line width=1.0pt, mark=diamond*,mark options={color=red}]
  table[row sep=crcr]{
  1.2627   0.042078 \\
  2.4436   0.011769 \\
  4.546   0.0030572 \\
  8.8342   0.00079154 \\
  17.8826   0.00019854 \\
};
\addlegendentry{$\mathbf{u}(\mathbf{x},t)$}

\addplot [color=blue,solid,line width=1.0pt,mark=square*,mark options={color=blue}]
  table[row sep=crcr]{
  1.2627   0.02891 \\
  2.4436   0.0039337 \\
  4.546   0.00072402 \\
  8.8342   0.00016135 \\
  17.8826   3.8876e-05 \\
};
\addlegendentry{$p(\mathbf{x},t)$}

\addplot [color=green,solid,line width=1.0pt,mark=triangle*,mark options={color=green}]
  table[row sep=crcr]{
  1.2627   0.043121 \\
  2.4436   0.0050789 \\
  4.546   0.00072301 \\
  8.8342   0.00012787 \\
  17.8826   2.5968e-05 \\
};
\addlegendentry{$T(\mathbf{x},t)$}

\addplot [color=black,solid,line width=0.5pt]
  table[row sep=crcr]{
 8.8342     0.0012 \\
 17.8826    0.0012 \\
 17.8826    0.0003 \\
 8.8342     0.0012 \\ 
};
\node[right, align=left, text=black, font=\footnotesize]
at (axis cs:17.8826,7.5727e-04) {2}; 

\end{axis}
\end{tikzpicture}
\end{subfigure}
\caption{Convergence test - nonlinear time-dependent problem - polynomial degree $\ell = 2$.}
\label{fig:nonlinearunsteadyconv}
\end{subfigure}
\end{figure}

\subsection{Robustness analysis}
\label{sec:robustness_analysis}
In order to investigate the method's robustness, we perform several convergence tests varying the model coefficients. We focus on the cases listed in Table~\ref{tab:linearsteady_coeff_robust}
with parameters $\alpha, \beta, \mu$ chosen as in Table~\ref{tab:TPE_params_convtest}. We recall that increasing the value of the dilatation coefficient $\lambda$ means going towards the incompressible limit. As space and time discretization parameters we consider the ones of Section~\ref{sec:convergence_analysis}. In Table~\ref{tab:linearsteady_robust_p1}, Table~\ref{tab:linearsteady_robust_p3}, and Table~\ref{tab:p2} we show the computed convergence rate in the $DG$-norm for $\mathbf{u}$ and in the $L^2$-norm for $p$ and $T$.
\begin{table}[H]
    \footnotesize
    \centering 
    \begin{tabular}{l | c | c | c | c }
    \textbf{Coefficient} & \textbf{Test $\mathbf{(i)}$} & \textbf{Test $\mathbf{(ii)}$} & \textbf{Test $\mathbf{(iii)}$} & \textbf{Test $\mathbf{(iv)}$} \\
    \hline \hline
    $a_0 \ [\si{\giga\pascal \per \kelvin\squared}]$ & 0 & 0.01 & 0 & 0 \\
    \hline
    $b_0 \ [\si{\per \kelvin}]$ & 0 & 0.01 & 0 & 0  \\
    \hline
    $c_0 \ [\si{\per \giga\pascal}]$ & 0 & 0.01 & 0 & 0  \\
    \hline
    $\lambda \ [\si{\giga\pascal}]$ & \num[exponent-product=\ensuremath{\cdot}]{5e6} & 5 & 5 & 5  \\
    \hline
    $\mathbf{K} \ [\si{\dm \squared \per \giga\pascal \per \hour}]$ & 0.2$\mathbf{I}$ & \num[exponent-product=\ensuremath{\cdot}]{2e-7}$\mathbf{I}$  & 0.2$\mathbf{I}$ & \num[exponent-product=\ensuremath{\cdot}]{2e-7}$\mathbf{I}$  \\
    \hline
    $\boldsymbol{\Theta} \ [\si{\dm \squared \giga\pascal \per \kelvin\squared \per \hour}]$ & 0.05$\mathbf{I}$  & \num[exponent-product=\ensuremath{\cdot}]{5e-8}$\mathbf{I}$  & \num[exponent-product=\ensuremath{\cdot}]{5e-8}$\mathbf{I}$ & 0.05$\mathbf{I}$  \\ 
    \hline
    \end{tabular}
    \\[5pt]
\caption{Problem's parameters for the robustness analysis (\textbf{Test $\mathbf{(iv)}$} is performed only for the non-linear case)}
\label{tab:linearsteady_coeff_robust}
\end{table}

First, comparing the results of $\textbf{Test (i)}$ with the error estimate of Theorem~\ref{thm:error_est_poinc}, we observe that our method is robust with respect to the quasi-incompressible case and the limit case in which the coefficients $a_0, b_0, c_0$ are very small or even equal to zero (cf. \cite{Brun2020}). The robustness with respect to $\lambda \gg 1$ is a key advantage of the proposed four-field formulation. 
Second, we observe in $\textbf{Test (ii)}$ and $\textbf{Test (iii)}$ that in both the linear and non-linear case we sometimes lose the $\ell+1$ order of accuracy for $p$ and $T$ when lowering the values of hydraulic mobility and thermal conductivity, respectively. Notice that, this results are in agreement with Theorem~\ref{thm:error_est_poinc} and that, by increasing the degree of approximation, we get better performances in terms of loss of accuracy. Looking again at the result of $\textbf{Test (ii)}$ and $\textbf{Test (iv)}$ we can notice that, even if the model is fully-coupled, the variations in $\textbf{K}$ affect only the order of accuracy of the pressure. On the other hand, in $\textbf{Test (iii)}$, we observe that decreasing the value of $\boldsymbol{\Theta}$ deteriorates the $\ell+1$ accuracy for the pressure. In our opinion this is due to the fact that the non-linear convective term becomes more significant for model coefficients as in $\textbf{Test (iii)}$. However, we observe that in all the tests we have carried out the convergence rate of the displacement error is never affected and the results are still in accordance with Theorem~\ref{thm:error_est_poinc}. We can conclude that the method is fully-robust with respect to the degenerate cases listed in Table~\ref{tab:linearsteady_coeff_robust}.
\begin{remark}
In the case $\mathbf{K}, \boldsymbol{\Theta} \ll 1$, we can achieve a global convergence of order $\ell$ by using the $(\ell+1, \ell, \ell, \ell)$ degree of approximation for $(\mathbf{u}, p, T, \varphi)$, respectively.
\end{remark}

\begin{table}[H]
    \footnotesize
    \centering 
    \begin{tabular}{c  c  c  c  c  c  c  c}
    Test & $1/h$ & $||e^{\textbf{u}}||_{DG}$ & $\text{roc}^u_{DG}$  & $||e^p||_{L^2}$ & $\text{roc}^p_{L^2}$  & $||e^T||_{L^2}$ & $\text{roc}^T_{L^2}$  \\
    \hline \hline
    \multirow{5}{*}{$\mathbf{(i)}$} 
    & $0.377$ & $2.882$ & $-$ & $0.201$ & $-$ & $0.223$ & $-$  \\
    & $0.214$ & $1.056$ & $1.779$ & $0.075$ & $1.745$ & $0.074$ & $1.961$  \\
    & $0.130$ & $0.415$ & $1.867$ & $0.022$ & $2.428$ & $0.025$ & $2.154$  \\
    & $0.074$ & $0.190$ & $1.398$ & $0.007$ & $2.158$ & $0.008$ & $1.939$  \\
    & $0.040$ & $0.100$ & $1.048$ & $0.002$ & $1.993$ & $0.003$ & $1.887$  \\
    \hline
    \multirow{5}{*}{$\mathbf{(ii)}$} 
    & $0.360$ & $2.851$ & $-$ & $0.177$ & $-$ & $0.177$ & $-$  \\
    & $0.217$ & $1.085$ & $1.912$ & $0.062$ & $2.069$ & $0.062$ & $2.069$  \\
    & $0.127$ & $0.420$ & $1.769$ & $0.017$ & $2.385$ & $0.017$ & $2.385$  \\
    & $0.070$ & $0.201$ & $1.241$ & $0.006$ & $1.885$ & $0.006$ & $1.885$  \\
    & $0.041$ & $0.104$ & $1.241$ & $0.002$ & $2.281$ & $0.002$ & $2.281$  \\
    \hline
    \multirow{ 5}{*}{$\mathbf{(iii)}$} 
    & $0.381$ & $2.844$ & $-$ & $0.199$ & $-$ & $0.223$ & $-$  \\
    & $0.221$ & $1.079$ & $1.781$ & $0.076$ & $1.766$ & $0.082$ & $1.842$  \\
    & $0.123$ & $0.436$ & $1.554$ & $0.022$ & $2.106$ & $0.026$ & $1.995$  \\
    & $0.078$ & $0.199$ & $1.707$ & $0.007$ & $2.631$ & $0.009$ & $2.360$  \\
    & $0.041$ & $0.102$ & $1.024$ & $0.002$ & $1.828$ & $0.003$ & $1.799$  \\
    \hline
    \multirow{ 5}{*}{$\mathbf{(iv)}$} 
    & $0.374$ & $2.875$ & $-$ & $0.225$ & $-$ & $0.231$ & $-$  \\
    & $0.218$ & $1.075$ & $1.829$ & $0.075$ & $2.048$ & $0.079$ & $1.993$  \\
    & $0.130$ & $0.428$ & $1.788$ & $0.022$ & $2.384$ & $0.025$ & $2.203$  \\
    & $0.073$ & $0.199$ & $1.325$ & $0.007$ & $1.985$ & $0.008$ & $1.914$  \\
    & $0.040$ & $0.103$ & $1.094$ & $0.002$ & $2.084$ & $0.003$ & $1.928$  \\
    \hline
    \end{tabular}
    \\[5pt]
\caption{Robustness test - linear steady problem - polynomial degree $\ell = 1$. The parameters are chosen as in Table~\ref{tab:linearsteady_coeff_robust}.}
\label{tab:linearsteady_robust_p1}
\end{table}

\begin{table}[H]
    \footnotesize
    \centering 
    \begin{tabular}{c  c  c  c  c  c  c  c}
    Test & $1/h$ & $||e^{\textbf{u}}||_{DG}$ & $\text{roc}^u_{DG}$  & $||e^p||_{L^2}$ & $\text{roc}^p_{L^2}$  & $||e^T||_{L^2}$ & $\text{roc}^T_{L^2}$  \\
    \hline \hline
    \multirow{5}{*}{$\mathbf{(i)}$} 
    & $0.377$ & $0.108$ & $-$ & \num[exponent-product=\ensuremath{\cdot}]{4.805e-4} & $-$ & \num[exponent-product=\ensuremath{\cdot}]{4.909e-4} & $-$  \\
    & $0.214$ & $0.018$ & $3.140$ & \num[exponent-product=\ensuremath{\cdot}]{4.402e-05} & $4.236$ & \num[exponent-product=\ensuremath{\cdot}]{4.383e-05} & $4.281$  \\
    & $0.130$ & $0.003$ & $3.629$ & \num[exponent-product=\ensuremath{\cdot}]{4.035e-06} & $4.780$ & \num[exponent-product=\ensuremath{\cdot}]{4.085e-06} & $4.747$  \\
    & $0.074$ & \num[exponent-product=\ensuremath{\cdot}]{5.373e-4} & $3.062$ & \num[exponent-product=\ensuremath{\cdot}]{3.816e-07} & $4.203$ & \num[exponent-product=\ensuremath{\cdot}]{3.926e-07} & $4.174$  \\
    & $0.040$ & \num[exponent-product=\ensuremath{\cdot}]{9.102e-05} & $2.910$ & \num[exponent-product=\ensuremath{\cdot}]{3.619e-08} & $3.861$ & \num[exponent-product=\ensuremath{\cdot}]{3.906e-08} & $3.782$  \\
    \hline
    \multirow{5}{*}{$\mathbf{(ii)}$} 
    & $0.360$ & $0.108$ & $-$ & \num[exponent-product=\ensuremath{\cdot}]{3.941e-4} & $-$ & \num[exponent-product=\ensuremath{\cdot}]{3.818e-4} & $-$  \\
    & $0.217$ & $0.018$ & $3.555$ & \num[exponent-product=\ensuremath{\cdot}]{3.741e-05} & $4.660$ & \num[exponent-product=\ensuremath{\cdot}]{3.447e-05} & $4.759$  \\
    & $0.127$ & $0.003$ & $3.249$ & \num[exponent-product=\ensuremath{\cdot}]{3.331e-06} & $4.509$ & \num[exponent-product=\ensuremath{\cdot}]{3.023e-06} & $4.537$  \\
    & $0.070$ & \num[exponent-product=\ensuremath{\cdot}]{5.439e-4} & $2.952$ & \num[exponent-product=\ensuremath{\cdot}]{3.388e-07} & $3.849$ & \num[exponent-product=\ensuremath{\cdot}]{3.104e-07} & $3.832$  \\
    & $0.041$ & \num[exponent-product=\ensuremath{\cdot}]{9.327e-05} & $3.320$ & \num[exponent-product=\ensuremath{\cdot}]{3.342e-08} & $4.361$ & \num[exponent-product=\ensuremath{\cdot}]{3.158e-08} & $4.303$  \\
    \hline
    \multirow{ 5}{*}{$\mathbf{(iii)}$} 
    & $0.381$ & $0.106$ & $-$ & \num[exponent-product=\ensuremath{\cdot}]{4.842e-4} & $-$ & \num[exponent-product=\ensuremath{\cdot}]{4.699e-4} & $-$  \\
    & $0.221$ & $0.019$ & $3.190$ & \num[exponent-product=\ensuremath{\cdot}]{4.390e-05} & $4.412$ & \num[exponent-product=\ensuremath{\cdot}]{4.288e-05} & $4.400$  \\
    & $0.123$ & $0.003$ & $3.105$ & \num[exponent-product=\ensuremath{\cdot}]{4.116e-06} & $4.054$ & \num[exponent-product=\ensuremath{\cdot}]{3.945e-06} & $4.086$  \\
    & $0.078$ & \num[exponent-product=\ensuremath{\cdot}]{5.546e-4} & $3.712$ & \num[exponent-product=\ensuremath{\cdot}]{3.945e-07} & $5.122$ & \num[exponent-product=\ensuremath{\cdot}]{4.118e-07} & $4.935$  \\
    & $0.041$ & \num[exponent-product=\ensuremath{\cdot}]{9.363e-05} & $2.710$ & \num[exponent-product=\ensuremath{\cdot}]{3.553e-08} & $3.667$ & \num[exponent-product=\ensuremath{\cdot}]{3.836e-08} & $3.615$  \\
    \hline
    \multirow{ 5}{*}{$\mathbf{(iv)}$} 
    & $0.373$ & $0.106$ & $-$ & \num[exponent-product=\ensuremath{\cdot}]{4.490e-4} & $-$ & \num[exponent-product=\ensuremath{\cdot}]{4.308e-4} & $-$  \\
    & $0.218$ & $0.018$ & $3.263$ & \num[exponent-product=\ensuremath{\cdot}]{4.278e-05} & $4.372$ & \num[exponent-product=\ensuremath{\cdot}]{4.204e-05} & $4.327$  \\
    & $0.130$ & $0.003$ & $3.451$ & \num[exponent-product=\ensuremath{\cdot}]{3.856e-06} & $4.671$ & \num[exponent-product=\ensuremath{\cdot}]{3.947e-06} & $4.592$  \\
    & $0.073$ & \num[exponent-product=\ensuremath{\cdot}]{5.502e-4} & $2.993$ & \num[exponent-product=\ensuremath{\cdot}]{3.856e-07} & $3.990$ & \num[exponent-product=\ensuremath{\cdot}]{4.053e-07} & $3.944$  \\
    & $0.040$ & \num[exponent-product=\ensuremath{\cdot}]{9.397e-05} & $2.937$ & \num[exponent-product=\ensuremath{\cdot}]{3.569e-08} & $3.955$ & \num[exponent-product=\ensuremath{\cdot}]{3.868e-08} & $3.904$  \\
    \hline
    \end{tabular}
    \\[5pt]
\caption{Robustness test - linear steady problem - polynomial degree $\ell = 3$. The parameters are chosen as in Table~\ref{tab:linearsteady_coeff_robust}.}
\label{tab:linearsteady_robust_p3}
\end{table}

\begin{table}[H]
    \footnotesize
    \centering 
    \begin{tabular}{c  c  c  c  c  c  c  c}
    Test & $1/h$ & $||e^{\textbf{u}}||_{DG}$ & $\text{roc}^u_{DG}$  & $||e^p||_{L^2}$ & $\text{roc}^p_{L^2}$  & $||e^T||_{L^2}$ & $\text{roc}^T_{L^2}$  \\
    \hline \hline
    \multirow{5}{*}{$\mathbf{(i)}$} 
    & $0.792$ & $0.041$ & $-$ & $0.001$ & $-$ & $0.002$ & $-$  \\
    & $0.409$ & $0.012$ & $1.930$ & \num[exponent-product=\ensuremath{\cdot}]{1.106e-4} & $3.667$ & \num[exponent-product=\ensuremath{\cdot}]{1.492e-4} & $3.636$  \\
    & $0.220$ & $0.003$ & $2.165$ & \num[exponent-product=\ensuremath{\cdot}]{1.054e-05} & $3.786$ & \num[exponent-product=\ensuremath{\cdot}]{1.439e-05} & $3.7672$  \\
    & $0.113$ & \num[exponent-product=\ensuremath{\cdot}]{7.817e-4} & $2.030$ & \num[exponent-product=\ensuremath{\cdot}]{1.112e-06} & $3.385$ & \num[exponent-product=\ensuremath{\cdot}]{1.712e-06} & $3.204$  \\
    & $0.056$ & \num[exponent-product=\ensuremath{\cdot}]{1.959e-4} & $1.962$ & \num[exponent-product=\ensuremath{\cdot}]{1.295e-07} & $3.049$ & \num[exponent-product=\ensuremath{\cdot}]{2.785e-07} & $2.575$  \\
    \hline
    \multirow{5}{*}{$\mathbf{(ii)}$} 
    & $0.792$ & $0.041$ & $-$ & $0.024$ & $-$ & $0.024$ & $-$  \\
    & $0.409$ & $0.011$ & $1.916$ & $0.004$ & $2.758$ & $0.004$ & $2.758$  \\
    & $0.220$ & $0.004$ & $2.165$ & \num[exponent-product=\ensuremath{\cdot}]{7.047e-4} & $2.747$ & \num[exponent-product=\ensuremath{\cdot}]{7.050e-4} & $2.747$  \\
    & $0.113$ & \num[exponent-product=\ensuremath{\cdot}]{7.882e-4} & $2.032$ & \num[exponent-product=\ensuremath{\cdot}]{1.617e-4} & $2.216$ & \num[exponent-product=\ensuremath{\cdot}]{1.620e-4} & $2.214$  \\
    & $0.056$ & \num[exponent-product=\ensuremath{\cdot}]{1.977e-4} & $1.961$ & \num[exponent-product=\ensuremath{\cdot}]{4.338e-05} & $1.866$ & \num[exponent-product=\ensuremath{\cdot}]{4.349e-05} & $1.864$  \\
    \hline
    \multirow{ 5}{*}{$\mathbf{(iii)}$} 
    & $0.792$ & $0.041$ & $-$ & $0.001$ & $-$ & $4.941$ & $-$ \\
    & $0.409$ & $0.012$ & $1.916$ & \num[exponent-product=\ensuremath{\cdot}]{1.112e-4} & $3.690$ & $0.082$ & $6.207$ \\
    & $0.220$ & $0.003$ & $2.165$ & \num[exponent-product=\ensuremath{\cdot}]{1.132e-05} & $3.680$ & $0.013$ & $2.947$ \\
    & $0.113$ & \num[exponent-product=\ensuremath{\cdot}]{7.882e-4} & $2.032$ & \num[exponent-product=\ensuremath{\cdot}]{1.786e-06} & $2.780$ & $0.001$ & $3.372$ \\
    & $0.056$ & \num[exponent-product=\ensuremath{\cdot}]{1.977e-4} & $1.961$ & \num[exponent-product=\ensuremath{\cdot}]{3.675e-07} & $2.242$ & \num[exponent-product=\ensuremath{\cdot}]{1.011e-4} & $3.729$ \\
    \hline
    \multirow{ 5}{*}{$\mathbf{(iv)}$} 
    & $0.792$ & $0.041$ & $-$ & $0.043$ & $-$ & $0.002$ & $-$  \\
    & $0.409$ & $0.012$ & $1.916$ & $0.007$ & $2.755$ & \num[exponent-product=\ensuremath{\cdot}]{1.492e-4} & $3.632$  \\
    & $0.220$ & $0.003$ & $2.165$ & $0.001$ & $2.745$ & \num[exponent-product=\ensuremath{\cdot}]{1.394e-05} & $3.818$  \\
    & $0.113$ & \num[exponent-product=\ensuremath{\cdot}]{7.882e-4} & $2.032$ & \num[exponent-product=\ensuremath{\cdot}]{2.911e-4} & $2.216$ & \num[exponent-product=\ensuremath{\cdot}]{1.473e-06} & $3.383$  \\
    & $0.056$ & \num[exponent-product=\ensuremath{\cdot}]{1.977e-4} & $1.961$ & \num[exponent-product=\ensuremath{\cdot}]{7.806e-05} & $1.867$ & \num[exponent-product=\ensuremath{\cdot}]{1.636e-07} & $3.116$  \\
    \hline
    \end{tabular}
    \\[5pt]
\caption{Robustness test - non-linear time-dependent problem - polynomial degree $\ell = 2$}
\label{tab:p2}
\end{table}

\subsection{Geothermal model problem}
\label{sec:geothermal}
We show now the results for a realistic model problem inspired by geothermal energy production. We consider a rectangular domain $\Omega = (0,4\si{\meter}) \times (0,1\si{\meter})$. As shown in Figure~\ref{fig:geothermal_compdom} we consider the domain to be a $2$D-slice of the subsoil and we build a polygonal mesh on it. We consider a uniform mesh made of $1000$ polygons with mesh size $h \sim 0.13 \si{\meter}$.
\begin{figure}[H]
 \includegraphics[width=1\textwidth]{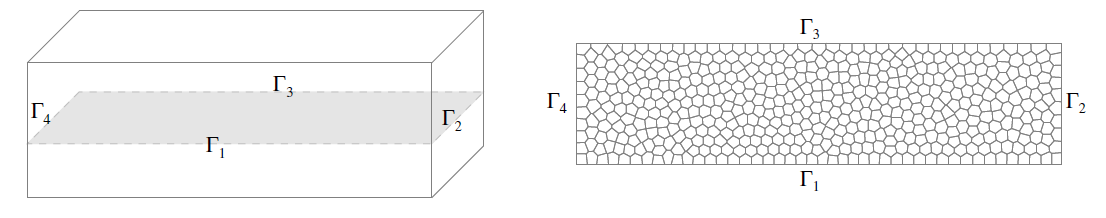}
\caption{Representation of the computational domain (left) and example of Voronoi mesh with $500$ elements (right).}
\label{fig:geothermal_compdom}
\end{figure}

In order to mimic the injection and extraction of a fluid (e.g. water) in the subsoil we impose Dirichlet boundary conditions on $\Gamma_2, \Gamma_4$. In the following description of the case test we will denote by $T_{\text{inj}}, p_{\text{inj}}, T_{\text{ext}}, p_{\text{ext}}$, the temperature and pressure of injection and extraction, respectively. We supplement our problem with zero loading terms $\mathbf{f}, g, h$, zero initial conditions $(\mathbf{u}_0, p_0, T_0)$, and with the following choice for the boundary conditions:
\begin{equation}
    \begin{aligned}
    & \boldsymbol{\sigma} \mathbf{n} = 0, \quad && \mathbf{K} \nabla p \cdot \mathbf{n}  = 0, \quad && \gamma (T-T_{\text{ext}}) + \boldsymbol{\Theta} \nabla T \cdot \mathbf{n} = 0 \quad && \text{on} \ \Gamma_1 \\
    & \mathbf{u} = 0, \quad && p = p_{\text{ext}}, \quad && T = T_{\text{ext}} \quad && \text{on} \ \Gamma_2 \\
    & \boldsymbol{\sigma} \mathbf{n} = 0, && \mathbf{K} \nabla p \cdot \mathbf{n} = 0, && \gamma (T-T_{\text{ext}}) + \boldsymbol{\Theta} \nabla T \cdot \mathbf{n} = 0 \quad && \text{on} \ \Gamma_3 \\
    & \mathbf{u} = 0, && p = p_{\text{inj}}, && T = T_{\text{inj}} \quad && \text{on} \ \Gamma_4 \\
    \end{aligned}
\end{equation}
where the parameter $\gamma$ is taken equal to $0.01$. We perform the simulations with a realistic choice of parameters, in particular they are taken identical to \cite{Brun2020}. We show the results up to $T_f = 3 \si{\hour}$ of simulation computed with $\Delta t = \num[exponent-product=\ensuremath{\cdot}]{5e-4} \si{h}$. The approximation degree is $\ell = 1$. We consider and compare two different scenarios where we vary the injection temperature. In both cases, we require an average of $3$ fixed point iterations in order to reach convergence.
\begin{figure}
\begin{subfigure}[t]{.49\textwidth}
    \centering
    \includegraphics[width=1\textwidth]{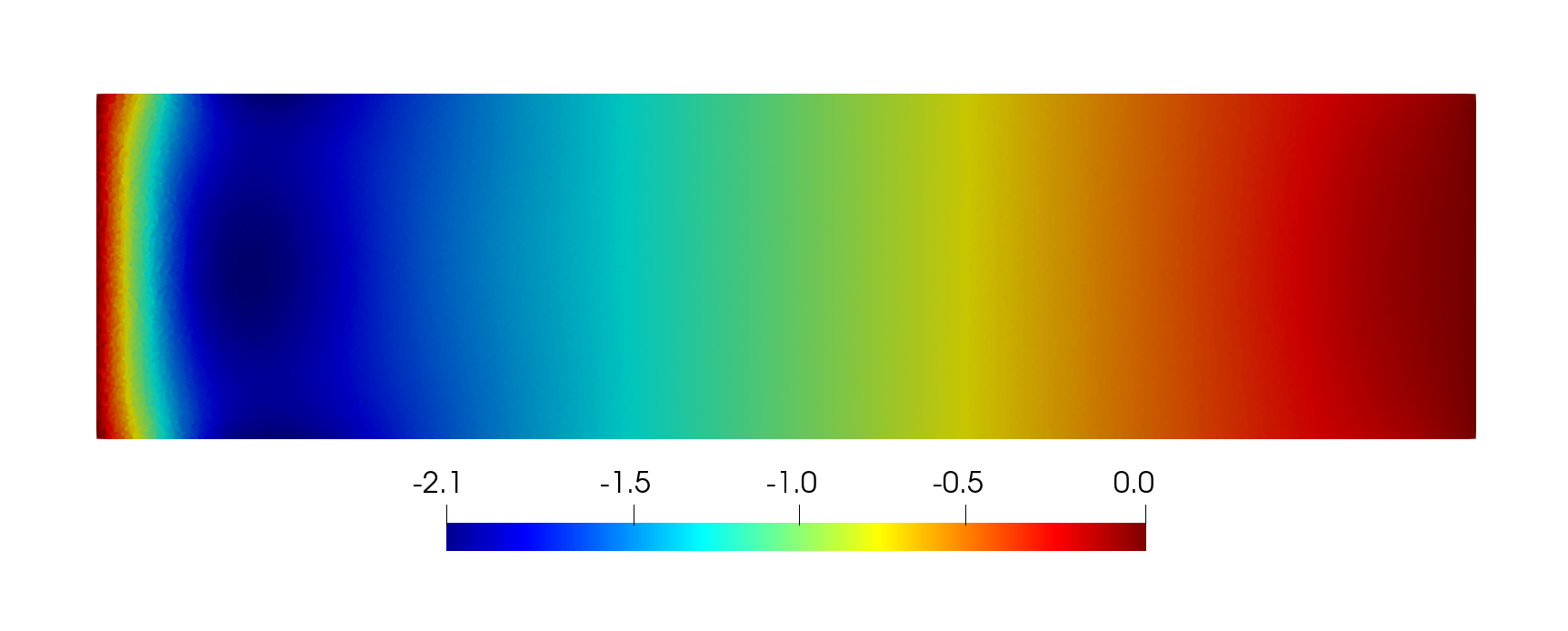}
    \caption{\footnotesize Numerical solution of the horizontal displacement $[\si{\cm}]$}
    \label{fig:u1_tinj60_pinj1}
\end{subfigure}
\begin{subfigure}[t]{.49\textwidth}
    \centering
    \includegraphics[width=1\textwidth]{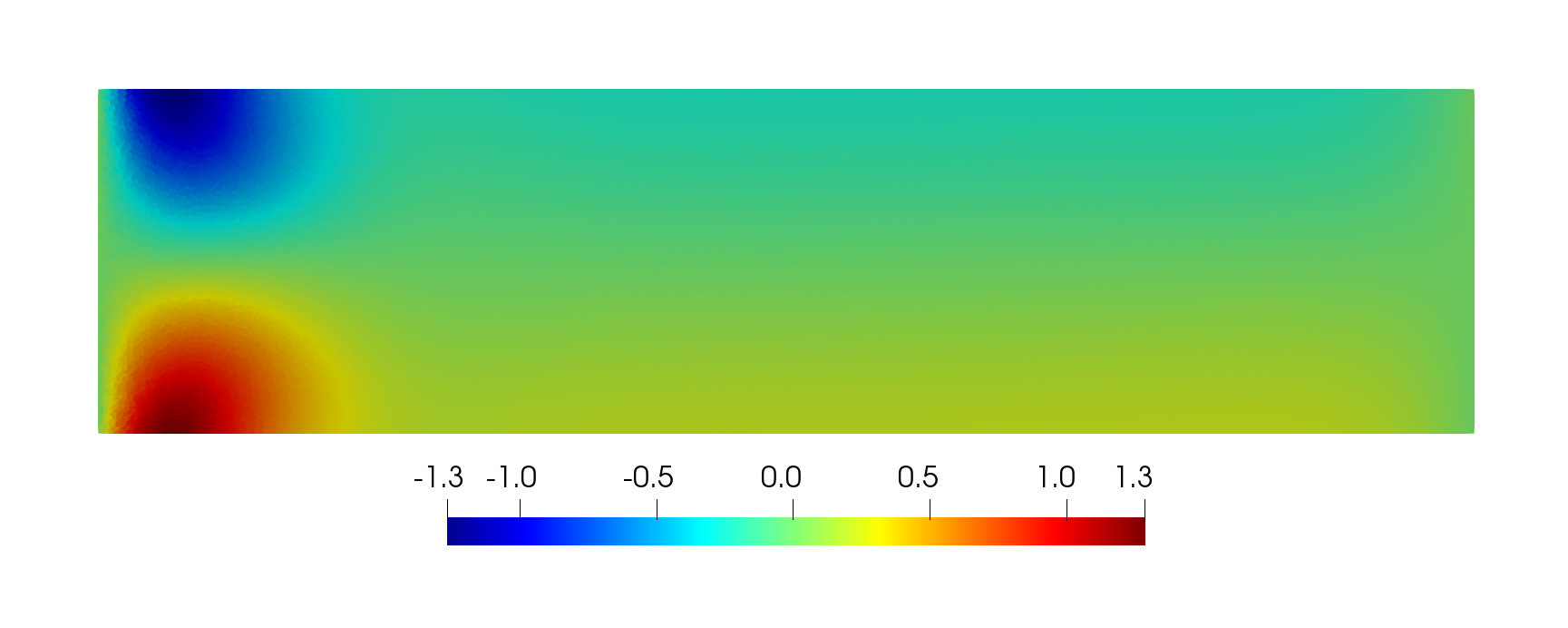}
    \caption{\footnotesize Numerical solution of the vertical displacement $[\si{\cm}]$}
    \label{fig:u2_tinj60_pinj1}
\end{subfigure}
\begin{subfigure}[b]{.49\textwidth}
    \centering
    \includegraphics[width=1\textwidth]{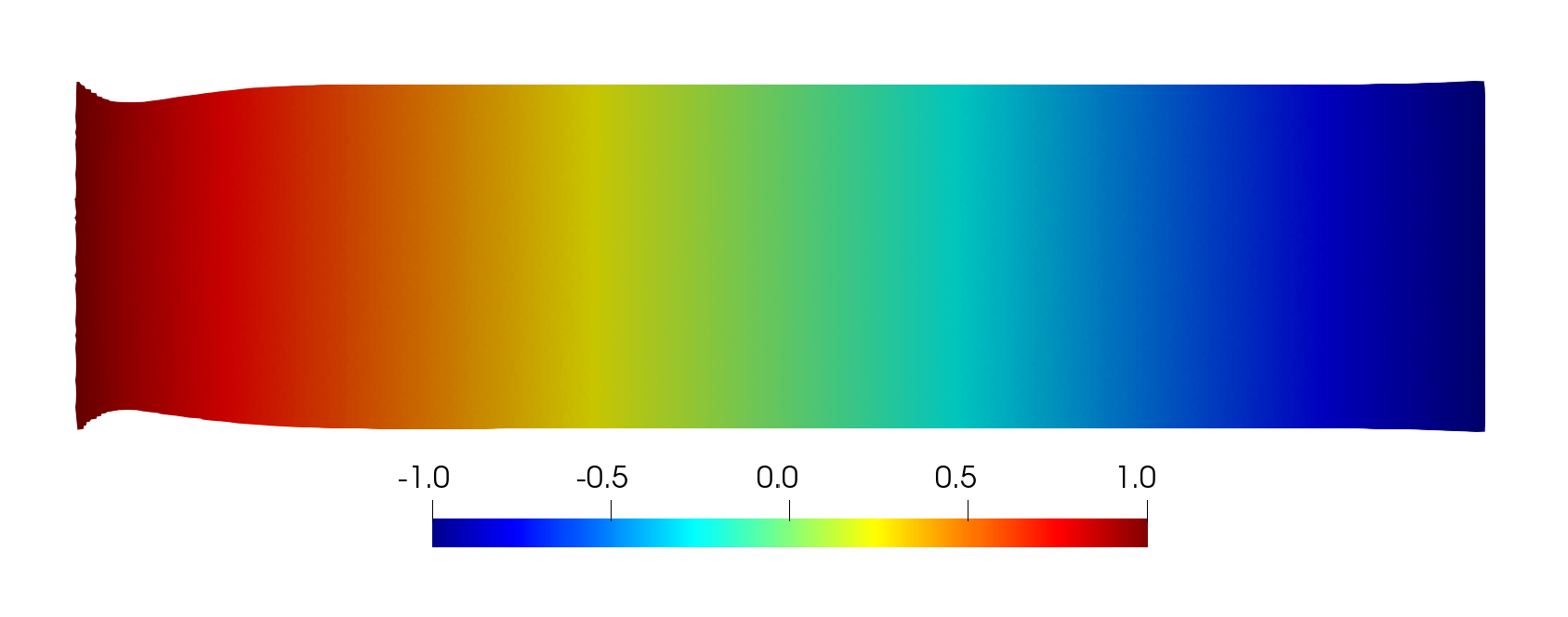}
    \caption{\footnotesize Numerical solution of the pressure $[\si{\mega \pascal}]$}
    \label{fig:Press_tinj60_pinj1}
\end{subfigure}
\begin{subfigure}[b]{.49\textwidth}
    \centering
    \includegraphics[width=1\textwidth]{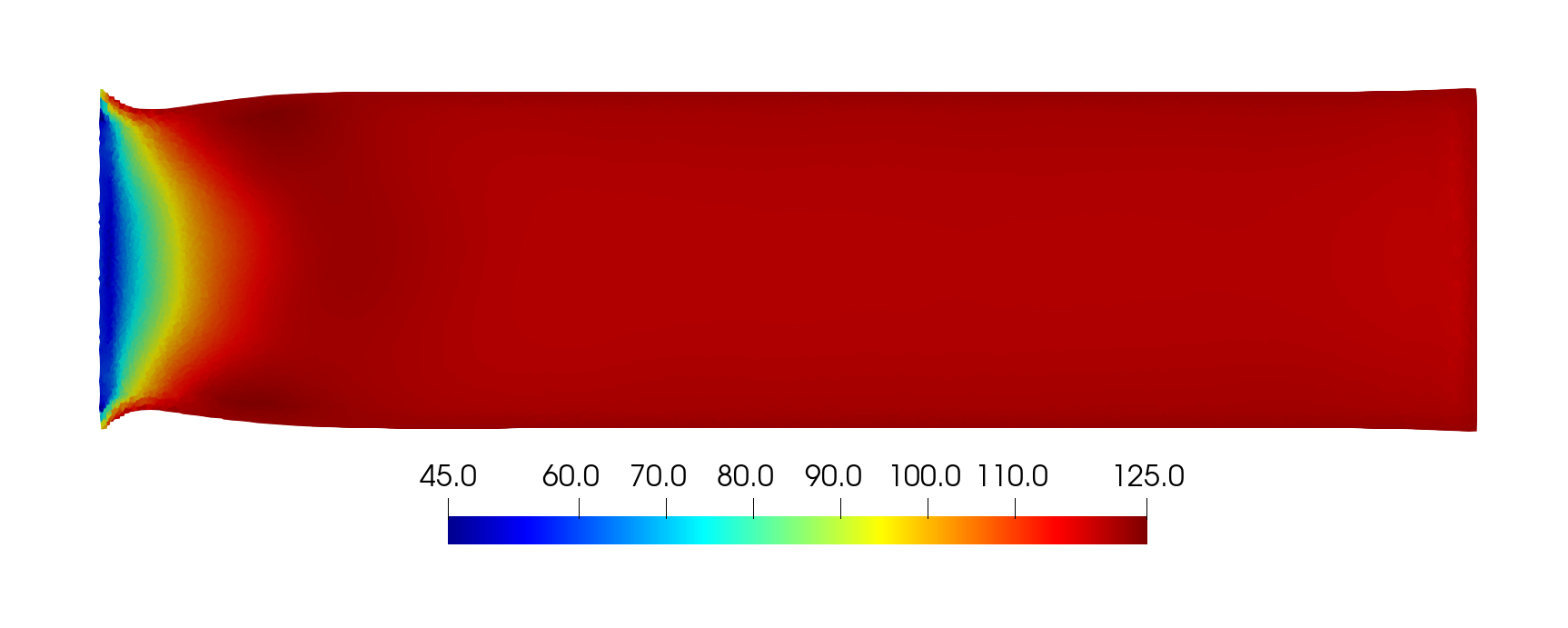}
    \caption{\footnotesize Numerical solution of the temperature $[\si{\celsius}]$}
    \label{fig:Temp_tinj60_pinj1}
\end{subfigure}
\caption{Simulation's parameters: $T_{\text{inj}} = 60 \si{\celsius}, \ \ T_{\text{ext}} = 120 \si{\celsius}, \ p_{\text{inj}} = 1 \si{\mega\pascal}, \  p_{\text{ext}} = -1 \si{\mega\pascal}$. (deformation magnified by a factor 5)}
\label{fig:geo_tinj60}
\end{figure}

\begin{figure}
\begin{subfigure}[t]{.49\textwidth}
    \centering
    \includegraphics[width=1\textwidth]{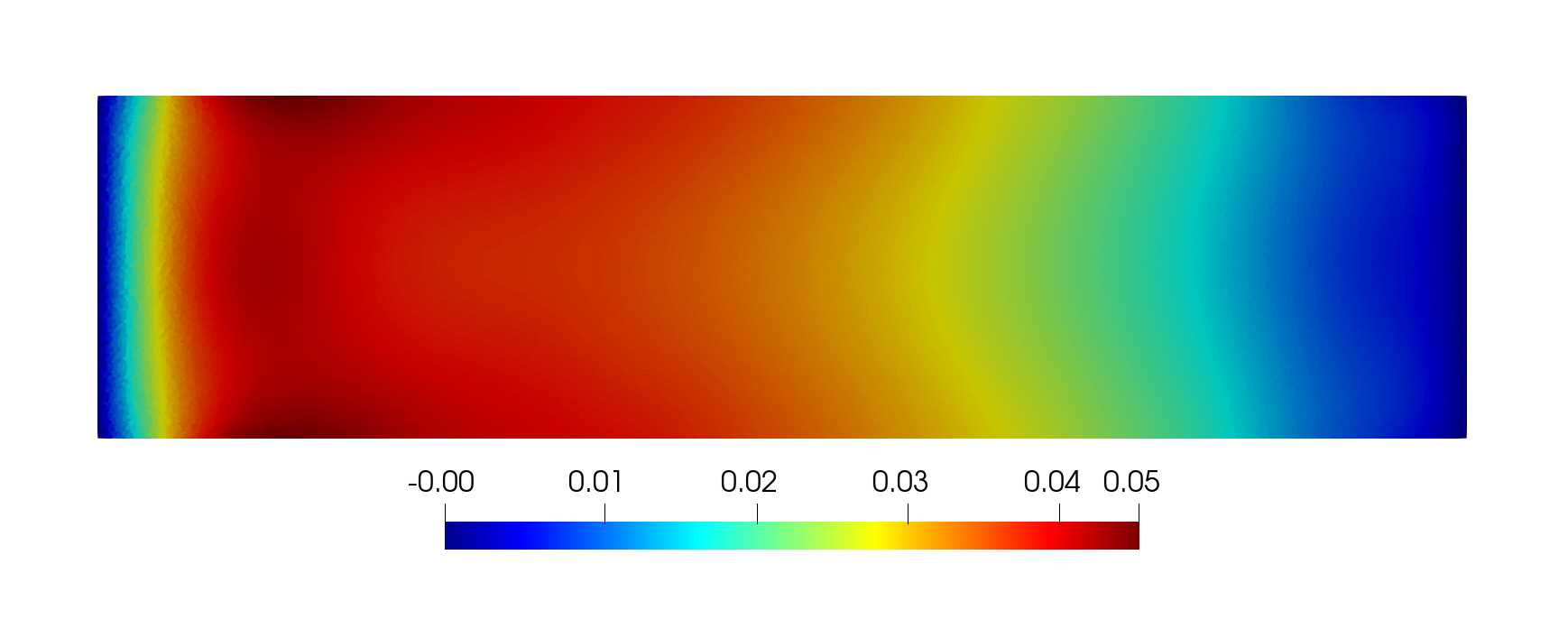}
    \caption{\footnotesize Numerical solution of the horizontal displacement $[\si{\mm}]$}
    \label{fig:u1_tinj120_pinj1}
\end{subfigure}
\begin{subfigure}[t]{.49\textwidth}
    \centering
    \includegraphics[width=1\textwidth]{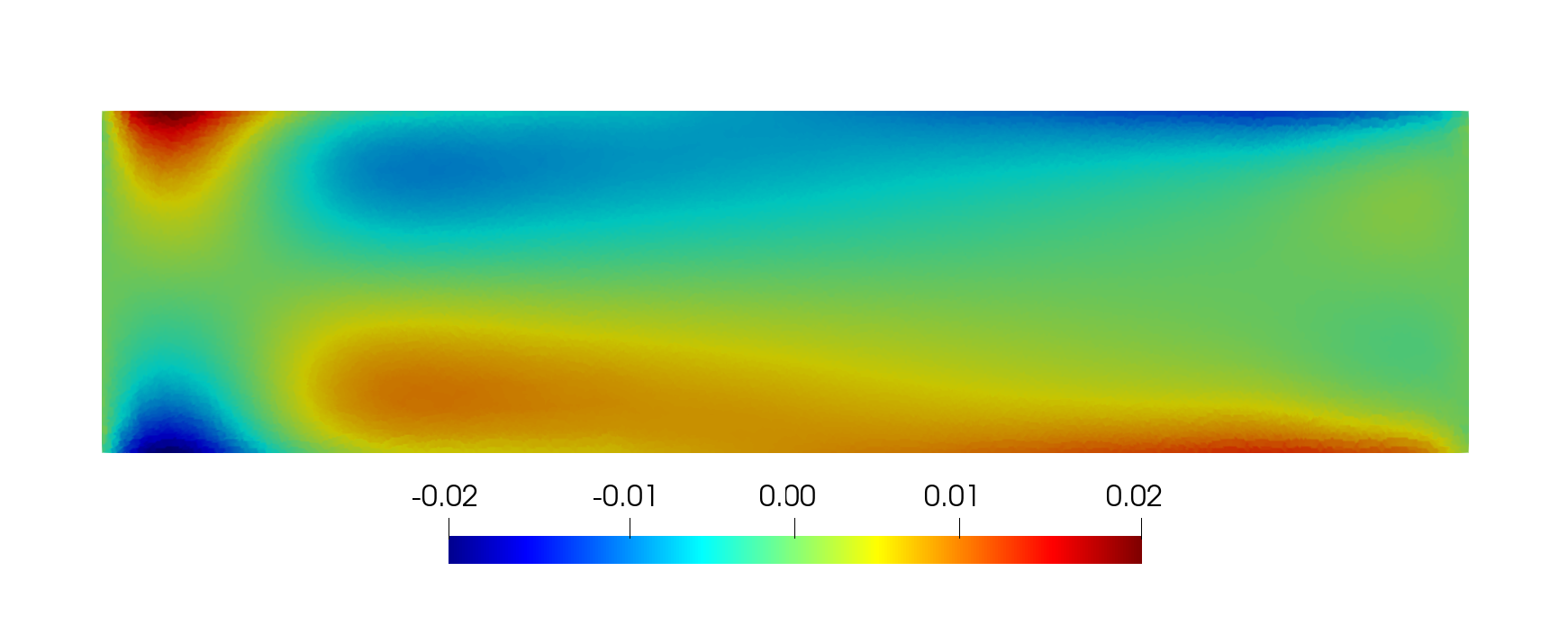}
    \caption{\footnotesize Numerical solution of the vertical displacement $[\si{\mm}]$}
    \label{fig:u2_tinj120_pinj1}
\end{subfigure}
\begin{subfigure}[b]{.49\textwidth}
    \centering
    \includegraphics[width=1\textwidth]{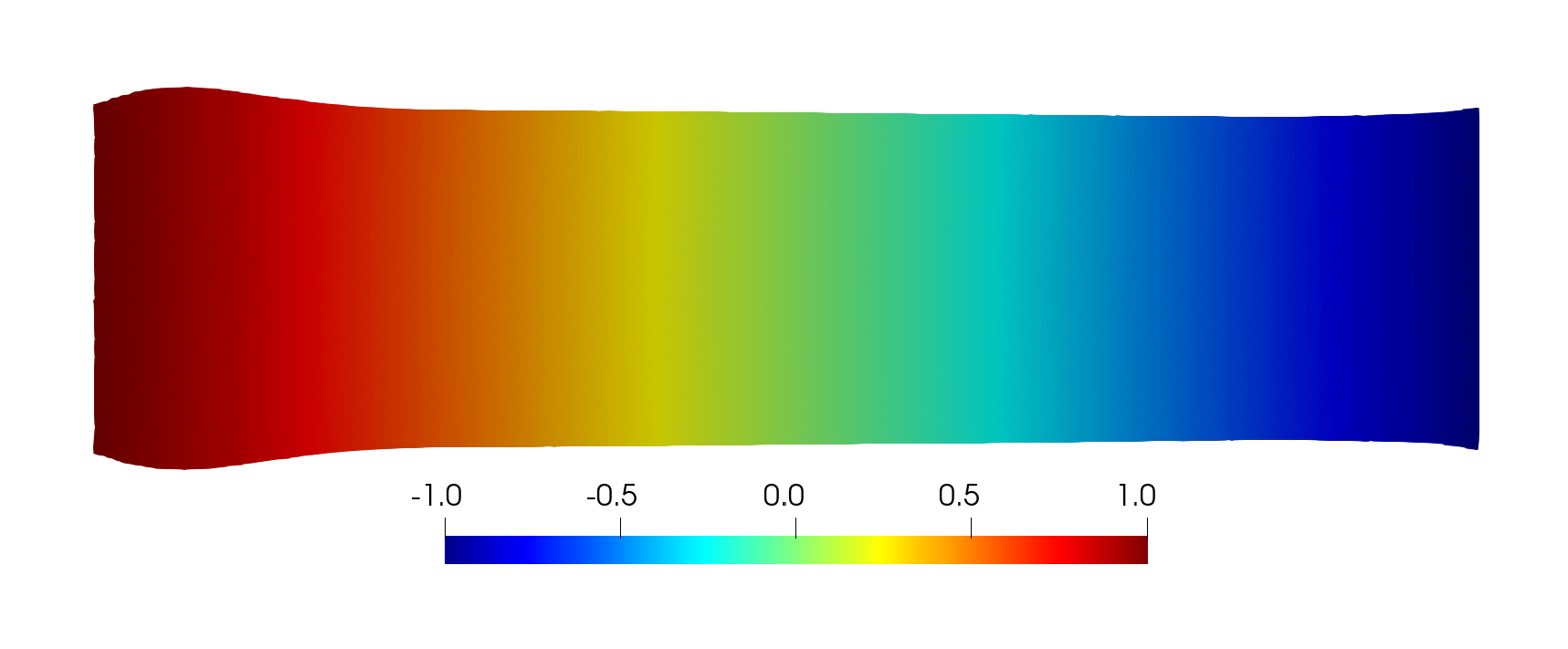}
    \caption{\footnotesize Numerical solution of the pressure $[\si{\mega \pascal}]$}
    \label{fig:Press_tinj120_pinj1}
\end{subfigure}
\begin{subfigure}[b]{.49\textwidth}
    \centering
    \includegraphics[width=1\textwidth]{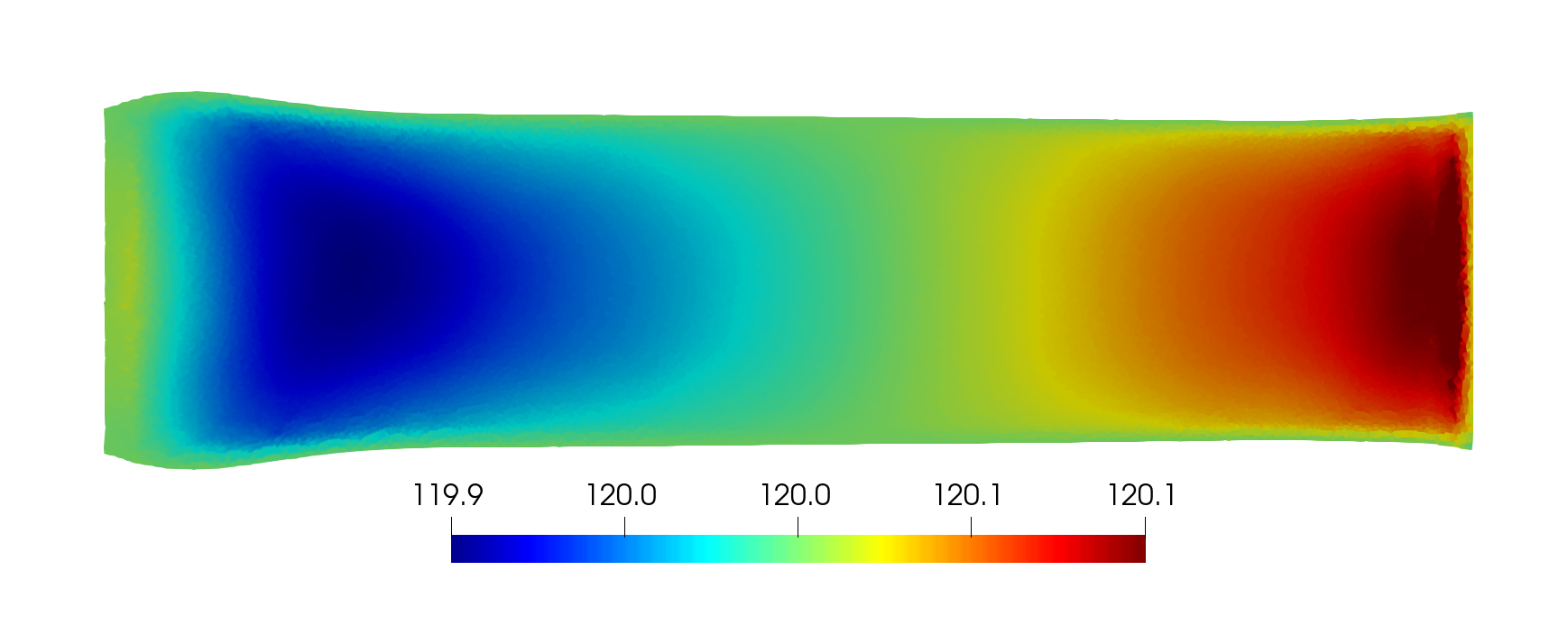}
    \caption{\footnotesize Numerical solution of the temperature $[\si{\celsius}]$}
    \label{fig:Temp_tinj120_pinj1}
\end{subfigure}
\caption{Simulation's parameters: $T_{\text{inj}} = 120 \si{\celsius}, \ \ T_{\text{ext}} = 120 \si{\celsius}, \ p_{\text{inj}} = 1 \si{\mega\pascal}, \  p_{\text{ext}} = -1 \si{\mega\pascal}$ (deformation magnified by a factor 2500)}
\label{fig:geo_tinj120}
\end{figure}

In Figure~\ref{fig:geo_tinj60} we show the displacement, pressure and temperature fields for a simulation with $T_{\text{inj}} = 60 \si{\celsius}$. First, we observe that the injected fluid is rapidly brought to the reference temperature by the system (cf. Figure~\ref{fig:Temp_tinj60_pinj1}). Second, we see that the injection of cold fluid results in a negative horizontal shift (Figure~\ref{fig:u1_tinj60_pinj1}) and a vertical shift that tends to tighten the region in the initial part of the domain (Figure~\ref{fig:u2_tinj60_pinj1}). Finally, we observe that the pressure decreases linearly along the domain (Figure~\ref{fig:Press_tinj60_pinj1}). In order to better understand the role of the temperature we compare the aforementioned results with the ones in Figure~\ref{fig:geo_tinj120} where the fluid is injected at $T_{\text{inj}} = T_{\text{ext}} = 120 \si{\celsius}$. We observe that, due to the pressure-temperature coupling, we have a slight variation of temperature in two macro-areas of the domain. In the initial part we have a swelling phenomenon (unlike the previous case), while in the second part of the domain the temperature variations manifest their effect. This can be observed particularly by comparing \ref{fig:u2_tinj120_pinj1} and \ref{fig:Temp_tinj120_pinj1}. Moreover, looking at the horizontal displacement depicted in Figure~\ref{fig:u1_tinj120_pinj1}, we notice that the domain tends to move towards the right, namely the opposite direction compared to Figure~\ref{fig:u1_tinj60_pinj1}. In the rightmost part of the domain the deformations due to fluid extraction, for which we would expect shrinking of the domain, are balanced by the slight temperature variation. To wrap up, we observe that with this choice of the parameters, the variations of temperature are more significant than the variations of pressure. This is confirmed by comparing the magnitude of the displacements in the two case tests.
We remark that, for the sake of representation, in Figure~\ref{fig:Temp_tinj60_pinj1}, Figure~\ref{fig:Press_tinj60_pinj1}, Figure~\ref{fig:Temp_tinj120_pinj1}, and Figure~\ref{fig:Press_tinj120_pinj1} the deformations of the domain are scaled by an appropriate factor (indicated in the figures' captions).

\section{Conclusions}
\label{sec:conclusion}
In this work, we have presented a four-field PolyDG formulation for the non-linear fully-coupled thermo-poroelastic problem. The stability estimate and error estimate in the semi-discrete framework are carried out highlighting the arbitrary-order property of the analyzed method. Numerical simulations are performed to test the convergence and robustness properties of the proposed method.
The results confirm the theoretical estimates and show that the method provides a good approximation also when considering limit cases in the ranges of physical parameters. Finally, a geothermal model problem is presented showing that, with appropriate choices of the parameters, the PolyDG discretization scheme can be appealing for real problems' simulations.

\bibliographystyle{abbrv}
\bibliography{bibliography}

\end{document}